\documentclass[11pt,a4paper]{amsart}
 
\usepackage{amsmath,amssymb}
\usepackage{amsthm} 
\usepackage[utf8]{inputenc} 
\usepackage{a4wide}
\usepackage{graphicx}
\usepackage{enumitem}

\newcommand{\UP}{\blacktriangle}
\newcommand{\DOWN}{\blacktriangledown}
\newcommand{\Up}{\vartriangle}
\newcommand{\Down}{\triangledown}

\theoremstyle{plain}
\newtheorem{theorem}{Theorem}[section]
\newtheorem{proposition}[theorem]{Proposition}
\newtheorem{lemma}[theorem]{Lemma}
\newtheorem{corollary}[theorem]{Corollary}

\theoremstyle{definition}

\newtheorem{example}[theorem]{Example}

\theoremstyle{remark}
\newtheorem{remark}[theorem]{Remark}

\setlist[itemize]{label=$\circ$, topsep=2pt}
\setlist[enumerate]{topsep=2pt}

\numberwithin{equation}{section}

\begin{document}

\title{The structure of rough sets defined by reflexive 
relations}

\author{Jouni J\"{a}rvinen}
\address[J.~J{\"a}rvinen]{Software Engineering, LUT School of Engineering Science, Mukkulankatu~19, 15210 Lahti, Finland}
\email{jouni.jarvinen@lut.fi}

\author{S\'{a}ndor Radeleczki}
\address[S.~Radeleczki]{Institute of Mathematics, University of Miskolc, 3515 Miskolc-Egyetemv\'{a}ros, Hungary}
\email{sandor.radeleczki@uni-miskolc.hu}

\keywords{Rough set, 
reflexive relation,
lattice structure, 
Dedekind--MacNeille completion, 
paraorthomodular lattice, 
completely join-prime element, 
completely join-irreducible element, 
core element, 
Kleene algebra, Nelson algebra,
quasiorder}

\subjclass{Primary 68T37,  06B23; Secondary 06D30,  03E20}

\begin{abstract}
For several types of information relations, the induced rough sets system RS does not form a lattice but only a partially ordered set. 
However, by studying its Dedekind--MacNeille completion DM(RS), one may
reveal new important properties of rough set structures. 
Building upon D.~Umadevi's work on describing joins and meets in DM(RS), we previously investigated pseudo-Kleene algebras defined on DM(RS) for reflexive relations. 
This paper delves deeper into the order-theoretic properties of DM(RS) in the context of reflexive relations. 
We describe the completely join-irreducible elements of DM(RS) and characterize when DM(RS) is a spatial completely distributive lattice.
We show that even in the case of a non-transitive reflexive relation, DM(RS) can form a Nelson algebra, a property generally associated with quasiorders.
We introduce a novel concept, the core of a relational neighborhood, and use it to provide a necessary and sufficient condition for DM(RS) to determine a Nelson algebra.
\end{abstract}

\maketitle

\hfill
Dedicated to the memory of Magnus Steinby

\section{Introduction}

Rough set theory, initially introduced in \cite{Pawl82}, is grounded in the concept of indistinguishability
relations. These relations are equivalence relations defined such that two objects are considered equivalent when they appear identical based on their observable properties.

Let $E$ be an equivalence relation on $U$.
Every subset $X$ of $U$ can be approximated by two sets.
The upper approximation $X^\UP$ consists of elements that are 
related to at least one element in $X$ by the equivalence relation 
$E$. This set represents the elements that might belong to $X$ given the information provided by $E$.
The lower approximation $X^\DOWN$ consists of elements such that all elements related to them by $E$ are also in $X$. This set represents the elements that definitely belong to $X$ based on the information provided by  $E$.

The rough set of $X$ is the pair $(X^\DOWN, X^\UP)$. This pair
captures the uncertainty related to the subset $X$ due to the limited information provided by the equivalence relation $E$. 
The collection of all rough sets, denoted by 
$\mathrm{RS}$, can be ordered coordinate-wise using the set-inclusion relation. The lattice-theoretic properties of RS have been extensively studied. For example, J.~Pomyka{\l}a and 
J.~A.~Pomyka{\l}a \cite{PomPom88} proved that RS forms a Stone 
algebra. This result was further refined by S.~D.~Comer
\cite{Come93}, who established that $\mathrm{RS}$ forms a regular
double Stone algebra. Piero Pagliani \cite{Pagl96} introduced a Nelson algebra structure on $\mathrm{RS}$ and showed that these 
Nelson algebras induced by equivalence relations are semisimple.

Rough approximations and rough sets can be extended beyond
equivalence relations by using other types of binary relations. 
A quasiorder, a reflexive and transitive relation, is one such 
example. Rough sets defined by quasiorders have been investigated 
by several researchers (see \cite{Ganter07, Kortelainen1994,
KumBan2012, KumBan2015, Umadevi13, Qiao2012}, for instance). 
The rough set lattice  $\mathrm{RS}$ determined by a quasiorder 
is known to be completely distributive and spatial. Moreover, a Nelson algebra can be defined on $\mathrm{RS}$ (see \cite{JPR13,
JR11, JR14, JarvRadeRivi24, JarRadVer09}). A tolerance, a reflexive and symmetric relation, is another type of binary relation that can be used to define rough sets. While the rough set system $\mathrm{RS}$ defined 
by a tolerance may not always form a lattice \cite{Jarvinen01}, 
it does form a completely distributive and spatial lattice when 
the tolerance is induced by an irredundant covering. Furthermore, 
a pseudocomplemented regular Kleene algebra can be defined on 
$\mathrm{RS}$ in this case (see \cite{Jarvinen_Radeleczki2014, JR2018, jarvinen_radeleczki_2018}).

For arbitrary binary relations D.~Umadevi presented the 
Dedekind--MacNeille completion $\mathrm{DM(RS)}$ of $\mathrm{RS}$ 
in \cite{Umadevi2015}. She also described the meets and joins in 
$\mathrm{DM(RS)}$.
In \cite{JarRad23}, we investigated pseudo-Kleene algebras defined 
on $\mathrm{DM(RS)}$ for reflexive relations. We proved that these
algebras are always paraorthomodular. 
Also in a recent work \cite{umadevi2024}, Umadevi studies
$\mathrm{DM(RS)}$ in case of a reflexive relation.

The motivation for this research is to further explore the order-theoretic and algebraic properties of relation-based rough set structures. Since some older studies have a special role, we start by recalling them.

J. Varlet \cite{Varlet72} proved that a double pseudocomplemented distributive lattice is regular if and only if every chain of its prime filters has at most two elements. Varlet \cite{Varlet66} also showed that a distributive double pseudocomplemented lattice is a regular double Stone algebra if and only if its prime filters form disjoint chains of one or two elements. In \cite{Cign86}, 
R.~Cignoli proved that a quasi-Nelson algebra (see page~\pageref{def:quasi-Nelson} of this work for the definition) is a Nelson algebra if and only if its prime filters have the so-called interpolation property.

In our previous works, we proved that in a complete, completely distributive, and spatial lattice, analogous results can be presented using completely join-irreducible elements. The advantage of this is that dealing with join-irreducible elements is simpler than dealing with prime filters.

We proved in \cite{jarvinen_radeleczki_2018} that a completely distributive and spatial pseudocomplemented Kleene algebra is regular if and only if its set of completely join-irreducible elements has at most two levels. Because the complete lattices of rough sets defined by tolerances induced by irredundant coverings are completely distributive and spatial, we applied this result to prove that $\mathrm{RS}$ forms a regular pseudocomplemented Kleene algebra. Using relationships between the completely 
join-irreducible elements, we also provided a representation theorem stating that each pseudocomplemented Kleene algebra defined on a completely distributive and spatial lattice is isomorphic to some $\mathrm{RS}$ defined by a tolerance induced by an irredundant covering.

In \cite{JR11}, we showed that a completely distributive and spatial Kleene algebra forms a Nelson algebra if and only if its set of completely join-irreducible elements has the interpolation property. Using this result, we proved that rough sets determined by a quasiorder form a Nelson algebra. Additionally, we proved that every completely distributive and spatial Nelson algebra is isomorphic to some rough set defined by a quasiorder.

Note also that M. Gehrke and E. Walker proved in \cite{GeWa92} that for an equivalence $E$, $\mathrm{RS}$ is isomorphic to $\mathbf{2}^I \times \mathbf{3}^K$, where $\mathbf{2}$ and $\mathbf{3}$ are the chains of 2 and 3 elements, $I$ is the set of the $E$-classes having one element, and $K$ is the set of $E$-classes having more than one element. This means that the completely join-irreducible
elements of $\mathrm{RS}$ form disjoint chains of one or two elements.

The above discussion highlights the special interest in the notions of completely join-irreducible elements, spatiality, and complete distributivity. In this work, we characterize the set of the completely join-irreducible elements of $\mathrm{DM(RS)}$. Furthermore, we characterize the cases when $\mathrm{DM(RS)}$ is spatial and completely distributive.

Our paper is structured as follows. 
In Section~\ref{Sec:BasicDefinitions}, we introduce the basic
notation. We also prove $\mathrm{DM(RS)}$ is a
subdirect product of $\wp(U)^\DOWN \times \wp(U)^\UP$,
where  $\wp(U)^\DOWN$ and $\wp(U)^\UP$ denote the complete 
lattices of all lower and upper approximations of subsets of $U$,
respectively.
In Section~\ref{Sec:Irreducibles&Atoms}, we provide a characterization
of the completely join-irreducible elements of $\wp(U)^\Up$ and 
$\wp(U)^\UP$. Note that $\wp(U)^\Down$ and $\wp(U)^\Up$ 
represent the approximations determined by the inverse 
$\breve{R}$ of $R$. For reflexive relations,
we characterize the completely join-irreducible elements
$\mathcal{J}$
of $\mathrm{DM(RS)}$ in terms of the completely join-irreducible elements of $\wp(U)^\Up$ and $\wp(U)^\UP$, and so-called
singletons. Singletons are the elements $x$ such that 
$|R(x)| =  1$.
Additionally, we give a formula for the atoms of $\mathrm{DM(RS)}$.

Section~\ref{Sec:Irreducibles&Atoms} also includes a couple of examples. Example~\ref{exa:JoinRS} demonstrates that there are reflexive 
relations for which $\mathrm{DM(RS)}$ forms a regular double Stone 
algebra. Since this property is typically associated with rough sets
defined by equivalence relations, this observation motivates our 
studies in Section~\ref{Sec:NelsonAlgebras}, where we characterize
those reflexive relations which determine rough set Nelson
algebras. 
Example~\ref{exa:no_join_irreducibles} illustrates that there are 
reflexive relations on $U$ such that the set of completely 
join-irreducible elements of $\wp(U)^\Up$ is empty. Consequently, 
the set of completely join-irreducible elements of $\mathrm{DM(RS)}$ 
is also empty.

In Section~\ref{Sec:CompleteDistributivitity}, we prove that 
$\mathrm{DM(RS)}$ is completely distributive if 
and only if any of $\wp(U)^\UP$, $\wp(U)^\DOWN$, $\wp(U)^\Up$, or 
$\wp(U)^\Down$ is completely distributive. For spatiality, we present
a similar characterization, showing that $\mathrm{DM(RS)}$ is spatial 
if and only if $\wp(U)^\UP$ and $\wp(U)^\Up$ are spatial. We  
characterize the case when $\mathrm{DM(RS)}$ is completely distributive and spatial in terms of the completely 
join-prime elements of $\wp(U)^\Up$ and $\wp(U)^\UP$. 
In this section, we also introduce the notion of the core of $R(x)$ and $\breve{R}(x)$, and
show that the completely join-prime elements of 
$\wp(U)^\Up$ are those sets $R(x)$ whose core is nonempty. 
A similar relationship holds between $\wp(U)^\UP$ and the sets 
$\breve{R}(x)$. 

In Section~\ref{Sec:NelsonAlgebras}, we shift our focus and assume 
that $\mathrm{DM(RS)}$ is completely distributive and spatial. We examine the structure of $\mathrm{DM(RS)}$ under these assumptions. 
We use cores of $R(x)$ and $\breve{R}(x)$ to introduce the condition
in Theorem~\ref{thm:characterization} equivalent to the fact
that the set $\mathcal{J}$ satisfies the
interpolation property. By a result appearing in \cite{JR11}
this condition is equivalent to the case that $\mathrm{DM(RS)}$ forms
a Nelson algebra.

\section{Rough sets by reflexive relations}
\label{Sec:BasicDefinitions}

Originally rough approximation operations were defined
in terms of equivalence relations \cite{Pawl82}.
Several studies in the literature define rough set approximations using arbitrary binary relations, with one of the earliest definitions appearing in \cite{Yao96}.

Let $R$ be a binary relation on $U$. We denote for
any $x \in U$, $R(x) := \{y \in U \mid (x,y) \in R \}$. 
The symbol $:=$ denotes ``equals by definition''.
For any set $X \subseteq U$, the \emph{lower approximation} of $X$ is
\[ X^\DOWN := \{x \in U \mid R(x) \subseteq X \} \]
and the \emph{upper approximation} of $X$ is
\[ X^\UP := \{x \in U \mid R(x) \cap X \neq \emptyset\}.\]

We may also determine rough set approximations in terms of the inverse
$\breve{R}$ of $R$, that is,
\[ X^\Down := \{x \in U \mid \breve{R}(x) \subseteq X \} \]
and
\[ X^\Up := \{x \in U \mid \breve{R}(x) \cap X \neq \emptyset \}. \]
For all $X \subseteq U$,
\[ X^\UP = \bigcup_{x \in X} \breve{R}(x) \text{ \ and \ }
 X^\Up = \bigcup_{x \in X} R(x) .\]
In particular, $\{x\}^\Up = R(x)$ and $\{x\}^\UP = \breve{R}(x)$ for all
$x \in U$.

Let $\wp(U)$ denote the family of all subsets of $U$. We also write 
\begin{gather*}
\wp(U)^\DOWN := \{ X^\DOWN \mid X \subseteq U\}, \qquad \wp(U)^\UP := \{ X^\UP \mid X \subseteq U\}, \\
\wp(U)^\Down:= \{ X^\Down \mid X \subseteq U\}, \qquad \wp(U)^\Up := \{ X^\Up \mid X \subseteq U\}. 
\end{gather*}

Galois connections appear in diverse areas of mathematics and theoretical computer science. A Galois connection is a relationship between two ordered sets that connects them through a pair of order-preserving functions, indicating a strong structural similarity. 

Let $P$ and $Q$ be ordered sets. A pair
$(f,g)$ of maps $f \colon P \to Q$ and $g \colon Q \to P$
is a \emph{Galois connection} between $P$ and $Q$ if
$f(p) \leq q \iff p \leq g(q)$ for all $p \in P$ and $q \in Q$.
The essential properties and results on Galois connections
can be found in \cite{Davey02,Erne_primer}, for example.

The pairs of maps $({^\UP},{^\Down})$ and $({^\Up},{^\DOWN})$ are 
Galois connections on the complete lattice $(\wp(U),\subseteq)$ as noted in \cite{Jarvinen2007}, for instance. This means that if $R$ is an arbitrary binary relation on
$U$, the following facts hold for all $X,Y \subseteq U$ and
$\mathcal{H} \subseteq \wp(U)$:

\begin{enumerate}[label={\rm (GC\arabic*)}]
    \item $\emptyset^\UP = \emptyset^\Up = \emptyset$ and
    $U^\DOWN = U^\Down = U$.
    \item $X^{\Down\UP} \subseteq X \subseteq X^{\UP\Down}$ and
    $X^{\DOWN\Up} \subseteq X \subseteq X^{\Up\DOWN}$.
    \item $X \subseteq Y$ implies $X^\DOWN \subseteq Y^\DOWN$, 
    $X^\Down \subseteq Y^\Down$, $X^\UP \subseteq Y^\UP$, 
    $X^\Up \subseteq Y^\Up$.
    \item $(\bigcup \mathcal{H})^\UP = 
    \bigcup \{X^\UP \mid X \in \mathcal{H}\}$
    and 
    $(\bigcup \mathcal{H})^\Up = \bigcup\{X^\Up \mid X \in \mathcal{H}\}$.
    \item $(\bigcap \mathcal{H})^\DOWN = 
    \bigcap \{X^\DOWN \mid X \in \mathcal{H}\}$
    and 
    $(\bigcap \mathcal{H})^\Down = \bigcap\{X^\Down \mid X \in \mathcal{H}\}$.
    \item $X^{\UP \Down \UP} = X^\UP$,
    $X^{\Up \DOWN \Up} = X^\Up$,
    $X^{\DOWN \Up \DOWN} = X^\DOWN$,
    $X^{\Down \UP \Down} = X^\Down$.
    \item $(\wp(U)^\UP,\subseteq) \cong (\wp(U)^\Down,\subseteq)$
     and 
    $(\wp(U)^\Up,\subseteq) \cong (\wp(U)^\DOWN,\subseteq)$.
\end{enumerate}
The operations $^\UP$ and $^\DOWN$ are mutually dual, and the same holds for $^\Up$ and $^\Down$, that is, 
for any $X \subseteq U$,
\begin{equation} \label{eq:dual}
X^{c \UP} = X^{\DOWN c},
X^{c \DOWN} = X^{\UP c},
X^{c \Up} = X^{\Down c}, 
X^{c \Down} = X^{\Up c}.
\end{equation}
Note that $X^c$ denotes the complement $U \setminus X$ of $X$. Because of the 
duality, $(\wp(U)^\UP,\subseteq) \cong (\wp(U)^\DOWN,\supseteq)$
and $(\wp(U)^\Up,\subseteq) \cong (\wp(U)^\Down,\supseteq)$. Therefore, by combining with (GC7), 
we can write the isomorphisms: 
\begin{equation}\label{eq:isomorphisms}
(\wp(U)^\UP,\subseteq) \cong (\wp(U)^\DOWN,\supseteq) \cong
(\wp(U)^\Up,\supseteq) \cong (\wp(U)^\Down,\subseteq).
\end{equation}
If $R$ is a \emph{reflexive} relation
on $U$, that is, $(x,x) \in R$ for all $x \in U$, then
\begin{enumerate}[label={\rm (Ref\arabic*)}]
\item $\emptyset^\DOWN = \emptyset^\Down = \emptyset$ and  
$U^\UP = U^\Up = U$.
\item $X^\DOWN \subseteq X \subseteq X^\UP$ and
$X^\Down \subseteq X \subseteq X^\Up$ for any $X \subseteq U$.
\end{enumerate}
\medskip%
Because $({^\Up}, {^\DOWN})$ is a Galois connection, 
$\wp(U)^\DOWN$ is a complete lattice such that 
\[
\bigwedge_{i \in I} X_i = \bigcap_{i \in I} X_i
\text{\quad and \quad}
\bigvee_{i \in I} X_i = 
\big ( \bigcup_{i \in I} X_i \big)^{\Up\DOWN}
\]
for all $\{X_i\}_{i \in I} \subseteq \wp(U)^\DOWN$.
Similarly, $\wp(U)^\UP$ is a complete lattice such that 
\[
\bigwedge_{i \in I} Y_i = 
\big ( \bigcap_{i \in I} Y_i \big )^{\Down \UP}
\text{\quad and \quad}
\bigvee_{i \in I} Y_i = 
\bigcup_{i \in I} Y_i 
\]
for any $\{Y_i\}_{i \in I} \subseteq \wp(U)^\UP$. These
facts imply that the Cartesian product 
\[ 
\wp(U)^\DOWN \times \wp(U)^\UP
= \{ (X,Y) \mid X \in \wp(U)^\DOWN \text{ and } 
Y \in \wp(U)^\UP \} 
\]
is a complete lattice such that
\begin{equation} \label{Eq:Meet}
\bigwedge \{  ( X_i, Y_i ) \mid i \in I \}  =
\Big ( \bigcap_{i \in I} X_i, \big ( \bigcap_{i \in I} Y_i \big )^{\Down \UP} \Big )
\end{equation}
and
\begin{equation} \label{Eq:Join}
\bigvee \{  ( X_i, Y_i ) \mid i \in I \}  =
\Big ( \big ( \bigcup_{i \in I} X_i \big )^{\Up \DOWN} , 
\bigcup_{i \in I} Y_i \Big ).
\end{equation}

Originally Pawlak \cite[p. 351]{Pawl82} defined a rough set as 
an equivalence class of sets which look the same in view of the knowledge restricted by the given indistinguishability relation, that is, as a class of sets having the same lower approximation and the same upper approximation. This notion generalizes in a 
natural way to arbitrary binary relations.

Let $R$ be a binary relation on $U$. A relation $\equiv$ is 
defined on $\wp(U)$ by
\[ 
X \equiv Y \iff X^\DOWN = Y^\DOWN \text{\quad and \quad}
X^\UP = Y^\UP.
\]
The equivalence classes of $\equiv$ are called \emph{rough sets}. Each element in a given rough set looks the same, when observed through the knowledge given by the knowledge $R$.
Namely, if $X \equiv Y$, then exactly the same elements belong certainly or possibly to $X$ and $Y$.

The order-theoretical study of rough sets was initiated by T.~B.~Iwi{\'n}ski in \cite{Iwinski87}. In his approach, rough sets are defined as pairs $(X^\DOWN, X^\UP)$, where $X \subseteq U$. This definition aligns with Pawlak's original definition, because if $\mathcal{C} \subseteq \wp(U)$ is an equivalence class of $\equiv$, then $\mathcal{C}$ is uniquely determined by the pair $(X^\DOWN, X^\UP)$, where $X$ is any member of $\mathcal{C}$; a set $Y \subseteq U$ belongs to $\mathcal{C}$ if and only if $(Y^\DOWN, Y^\UP) = (X^\DOWN, X^\UP)$.

We denote by $\mathrm{RS}$ the set of all rough sets, that is,
\[ \mathrm{RS} := \{ (X^\DOWN, X^\UP) \mid X \subseteq U \}. \]
The set $\mathrm{RS}$ can be ordered coordinatewise by
\[ (X^\DOWN, X^\UP) \leq (Y^\DOWN, Y^\UP) \overset{\rm def}{\iff} 
X^\DOWN \subseteq Y^\DOWN \ \text{and}  \ X^\UP \subseteq Y^\UP. \]
It is known that $\mathrm{RS}$ is not always a lattice if $R$ is a 
reflexive and symmetric binary relation; see \cite{Jarvinen01}.

The \emph{Dedekind--MacNeille completion} of an ordered set is the smallest complete lattice containing it; see \cite{Davey02},
for example.
We denote the Dedekind--MacNeille completion of $\mathrm{RS}$ by $\mathrm{DM(RS)}$. 
D.~Umadevi \cite{Umadevi2015} has proved that for any binary relation $R$ on $U$,
\[ \mathrm{DM(RS)} = \{ (A,B) \in \wp(U)^\DOWN \times \wp(U)^\UP \mid A^{\Up \UP} \subseteq B \text{ and }  A \cap \mathcal{S} = B \cap \mathcal{S} \} .\]
Here 
\[ \mathcal{S} := 
\{ x \in U \mid R(x) = \{z\} \text{ for some $z \in U$} \}. \]
The elements of $\mathcal{S}$ are called \emph{singletons}.
This means that $x \in \mathcal{S}$ if and only if
$|R(x)| = 1$.
The following proposition, appearing in \cite{Umadevi2015}, describes how meets and joins are formed in the complete lattice $\mathrm{DM(RS)}$.

\begin{proposition} \label{prop:uma_operations}
The meets and joins are formed in $\mathrm{DM(RS)}$ as in
\eqref{Eq:Meet} and \eqref{Eq:Join}, respectively.
\end{proposition}

Next we present two examples, which motivate the study of rough
sets defined by reflexive relations, which are the main topic
of this article.

\begin{example} In this work, we primarily consider rough sets defined by 
reflexive relations. Reflexive relations can be viewed as \emph{directional similarity relations}. 
Reflexivity is an evident property of similarity, as each object is inherently similar to itself.

Symmetry is not an obvious property of similarity. For instance, A.~Tversky states in \cite{Tversky77} that similarity should not be treated as a symmetric relation. Statements like ``$a$ is like $b$'' are directional, with $a$ as the subject and $b$ as the referent. It is not equivalent in general to the converse similarity statement ``$b$ is like $a$''. Tversky also provides concrete examples, like ``the portrait resembles the person'' rather than ``the person resembles the portrait'', and ``the son resembles the father'' rather than ``the father resembles the son''.

It is also clear that similarity relations are not necessarily transitive. For example, we may have a finite sequence of objects $x_1, x_2, \ldots, x_n$ such that each pair of consecutive objects $x_i$ and $x_{i+1}$ are similar, but the objects $x_1$ and $x_n$ have nothing in common.
\end{example}

In the following example, we discuss how rough approximations, rough sets, and their operations can be viewed when defined in terms of a directional similarity relation.

\begin{example} \label{Exa:set_operations}
Let $R$ be a directional similarity relation on $U$. If $X$ is a subset of $U$, then $X^\DOWN$ consists of elements $x$ such that all elements to which $x$ is $R$-similar are in $X$. The set $X^\DOWN$ can therefore be viewed as the set of elements that are certainly in $X$ given the knowledge $R$. Similarly, the set 
$X^\UP$ can be seen as the set of elements that are possibly in $X$. If $x$ is in $X^\UP$, then there is at least one element in 
$X$ to which $x$ is $R$-similar. The rough set $(X^\DOWN,X^\UP)$ 
of $X$ represents how the set $X$ appears when observed through the knowledge restricted by $R$

From usual sets, we can form new sets using the operations of union (elements that belong to at least one of the sets), intersection (elements that belong to all sets), and complement (elements that do not belong to the set). Since $\mathrm{DM(RS)}$ forms a lattice with a De Morgan complementation $\sim$, we can have lattice-theoretical counterparts of these set operations.

The join of two rough sets $(X^\DOWN, X^\UP)$ and $(Y^\DOWN, Y^\UP)$, corresponding to the union, can be computed in $\mathrm{DM(RS)}$ as:
\[
(X^\DOWN,X^\UP) \vee  (Y^\DOWN, Y^\UP) =
( (X^\DOWN \cup Y^\DOWN)^{\Up\DOWN}, (X^\UP \cup Y^\UP)) =
((X^{\DOWN \Up} \cup Y^{\DOWN \Up})^{\DOWN}, (X \cup Y)^\UP).
\]
Similarly, their intersection can be considered as the meet
in $\mathrm{DM(RS)}$:
\[
(X^\DOWN,X^\UP) \wedge  (Y^\DOWN, Y^\UP) =
( X^\DOWN \cap Y^\DOWN, (X^\UP \cap Y^\UP)^{\Down\UP} ) =
((X \cap Y)^\DOWN, (X^{\UP\Down} \cup Y^{\UP\Down})^\UP).
\]
The complement may be defined as:
\[
{\sim}(X^\DOWN,X^\UP) = ( X^{\UP c}, X^{\DOWN c}) =
(X^{c \DOWN}, X^{c \UP}).
\]
\end{example}

A \emph{subdirect product} $L$ of an indexed family of
complete lattices $\{L_i\}_{i \in I}$ is a complete sublattice of the direct product 
$\prod_{i \in I} L_i$ such that the canonical projections 
$\pi_i$ are all surjective, that is, $\pi_i(L) = L_i$
for all $i \in I$ \cite{Balb74,Birkhoff79}.

\begin{proposition} \label{Prop:Completion}
If $R$ is a binary relation on $U$, then $\mathrm{DM(RS)}$ is a subdirect product of  $\wp(U)^\DOWN$ and $\wp(U)^\UP$.
\end{proposition}

\begin{proof}
By Proposition~\ref{prop:uma_operations}, 
$\mathrm{DM(RS)}$ is a complete sublattice of 
$\wp(U)^\DOWN \times \wp(U)^\UP$, because joins and meets are formed analogously.

The maps $\pi_{1} \colon (X^\DOWN, Y^\UP) \mapsto X^\DOWN$ and
$\pi_{2} \colon (X^\DOWN, Y^\UP) \mapsto Y^\UP$ are the canonical projections 
of the product  $\wp(U)^\DOWN \times \wp(U)^\UP$. Obviously,
their restrictions to $\mathrm{DM(RS)}$ are surjective. This can be
seen, for instance, because $\mathrm{RS} \subseteq \mathrm{DM(RS)}$.
If $X \in \wp(U)^\DOWN$, then there is $Y \subseteq U$ such that
$X = Y^\DOWN$. Now $\pi_1 (Y^\DOWN,Y^\UP) = Y^\DOWN = X$. The proof concerning
$\pi_2$ is analogous.
\end{proof}

We end this section by presenting the following two useful
lemmas.

\begin{lemma} \label{lem:helper}
Let $R$ be a binary relation on $U$.
If $(A_i,B_i)_{i \in I}$ and $(C_j,D_j)_{j \in J}$ are subsets
of $\mathrm{DM(RS)}$, then
\[
{\bigvee}_{i \in I} (A_i,B_i) \vee {\bigvee}_{j \in J} (C_j,D_j) =
\Big ( 
\big ( {\bigcup}_{i \in I} A_i \cup {\bigcup}_{j \in J} C_j 
\big )^{\Up \DOWN}, 
\bigcup_{i \in I} B_i \cup \bigcup_{j \in J} D_j 
\Big ).
\]
\end{lemma}

\begin{proof} In this proof, equalities (i) and (ii) are based on how the joins are formed and follow from \eqref{Eq:Join}. Because 
$^\Up$ distributes over joins, we have (iii) in view of (GC4). (iv) is a plain application of (GC6). Equality (v) is again 
based on (GC4).

\begin{gather*}
{\bigvee}_{i \in I} (A_i,B_i) \vee {\bigvee}_{j \in J} (C_j,D_j) 
\stackrel{\rm (i)}{=} 
\Big (  \big ( {\bigcup}_{i \in I} A_i \big )^{\Up\DOWN}, \
{\bigcup}_{i \in I} B_i \Big) \vee 
\Big (\big ( {\bigcup}_{j \in J} C_j \big )^{\Up\DOWN}, \
{\bigcup}_{j \in J} D_j \Big ) \\
\stackrel{\rm (ii)}{=} 
\Big ( \big (  \big ( {\bigcup}_{i \in I} A_i \big )^{\Up\DOWN} \cup
\big ( {\bigcup}_{j \in J} C_j \big )^{\Up\DOWN} \big )^{\Up \DOWN}, \
{\bigcup}_{i \in I} B_i \cup {\bigcup}_{j \in J} D_j \Big ) \\
\stackrel{\rm (iii)}{=} 
\Big ( \big (  \big ( {\bigcup}_{i \in I} A_i \big )^{\Up\DOWN \Up} \cup
\big ( {\bigcup}_{j \in J} C_j \big )^{\Up\DOWN \Up} \big )^{\DOWN}, \
{\bigcup}_{i \in I} B_i \cup {\bigcup}_{j \in J} D_j \Big ) \\
\stackrel{\rm (iv)}{=} 
\Big ( \big (  \big ( {\bigcup}_{i \in I} A_i \big )^{\Up} \cup
\big ( {\bigcup}_{j \in J} C_j \big )^{\Up} \big )^{\DOWN}, \
{\bigcup}_{i \in I} B_i \cup {\bigcup}_{j \in J} D_j \Big ) \\
\stackrel{\rm (v)}{=} 
\Big ( \big (  {\bigcup}_{i \in I} A_i  \cup
{\bigcup}_{j \in J} C_j \big )^{\Up \DOWN}, \
{\bigcup}_{i \in I} B_i \cup {\bigcup}_{j \in J} D_j \Big ) .
\qedhere
\end{gather*}
\end{proof}

\begin{lemma}\label{lem:rule}
Let $R$ be a reflexive relation. If $(A,B) \in \mathrm{DM(RS)}$, then
\begin{equation} \label{eq:alt_form}
 (A,B) = \bigvee \{ (\{a\}^{\Up \DOWN}, \{a\}^{\Up \UP} ) \mid a \in A \}
\vee \bigvee \{ (\{b\}^\DOWN,\{b\}^\UP) \mid b \in B^\Down \}.
\end{equation}
\end{lemma}

\begin{proof}
Using Lemma~\ref{lem:helper}, the right side of Equation \eqref{eq:alt_form} can be written 
in the form:
\[
\Big ( \big ( \bigcup \{ \{a\}^{\Up\DOWN} \mid a \in A\} \cup
\bigcup \{ \{b\}^\DOWN \mid b \in B^\Down\} \big)^{\Up\DOWN}, \
 \bigcup \{ \{a\}^{\Up\UP} \mid a \in A\} \cup
 \bigcup \{ \{b\}^\UP \mid b \in B^\Down \}
\Big ). 
\]
Let us denote this pair by $(\mathsf{L},\mathsf{R})$. 
We prove that $(A,B) = (\mathsf{L},\mathsf{R})$.

First, we show that $B = \mathsf{R}$. 
If $a \in A$, then $\{a\} \subseteq A$ and 
$\{a\}^{\Up \UP} \subseteq A^{\Up \UP} \subseteq B$. 
If $b \in B^\Down$,
then $\{b\}^\UP \subseteq B^{\Down \UP} = B$. This means that 
$\mathsf{R} \subseteq B$.

Suppose that 
$c \in B$. Because $B = B^{\Down\UP} = 
\bigcup \{ \{b\}^\UP \mid b \in B^\Down\}$,
this means that $c \in \{b\}^\UP$ for some $b \in B^\Down$. Thus,
$c \in R$. We have proved that $\mathsf{R} = B$.

To prove $A = \mathsf{L}$, we show first that
\[ \bigcup \{ \{b\}^\DOWN \mid b \in B^\Down\} \subseteq
\bigcup \{ \{a\}^{\Up\DOWN} \mid a \in A\}.
\]
There are two possibilities: $\{b\}^\DOWN = \emptyset$ and
$\{b\}^\DOWN \neq \emptyset$. If $\{b\}^\DOWN = \emptyset$ for some
$b \in B^\Down$, then $\bigcup \{ \{b\}^\DOWN \mid b \in B^\Down\}$ 
does not change because of this $b$. If $\{b\}^\DOWN \neq \emptyset$,
then this means that $\{b\}^\DOWN = \{b\}$ and $R(b) = \{b\}$. We
have that $b \in \mathcal{S}$. Hence, 
$b \in B \cap \mathcal{S} = A \cap \mathcal{S} \subseteq A$.
Because $\{b\} = R(b) = \{b\}^\Up$, we have that
$\{b\}^\DOWN = \{b\}^{\Up\DOWN}$ belongs to 
$\{ \{a\}^{\Up\DOWN} \mid a \in A\}$ and we have again that
$\bigcup \{ \{b\}^\DOWN \mid b \in B^\Down\} \subseteq
\bigcup \{ \{a\}^{\Up\DOWN} \mid a \in A\}$. Therefore,

\[ \bigcup \{ \{a\}^{\Up\DOWN} \mid a \in A\} \cup
\bigcup \{ \{b\}^\DOWN \mid b \in B^\Down\} = 
\bigcup \{ \{a\}^{\Up\DOWN} \mid a \in A\} . \]

Let $b \in \bigcup \{ \{a\}^{\Up\DOWN} \mid a \in A\}$. Then
$b \in \{a\}^{\Up\DOWN}$ for some $a \in A$ and
$b \in \{a\}^{\Up\DOWN} \subseteq A^{\Up\DOWN} = A$. Thus, 
$\bigcup \{ \{a\}^{\Up\DOWN} \mid a \in A\} \subseteq A$ which gives 
$\big ( \bigcup \{ \{a\}^{\Up\DOWN} \mid a \in A\} \big )^{\Up\DOWN} \subseteq 
A^{\Up \DOWN} = A$. 
This means that $\mathsf{L} \subseteq A$.

On the other hand, if $a \in A$, then $a \in \{a\} \subseteq
\{a\}^{\Up \DOWN}$. Thus,
\[ a \in\bigcup \{ \{a\}^{\Up\DOWN} \mid a \in A\} 
\subseteq (\bigcup \{ \{a\}^{\Up\DOWN} \mid a \in A\})^{\Up \DOWN}.\]
We have proved that $A \subseteq \mathsf{L}$ and also 
$\mathsf{L} = A$. This means
that equation \eqref{eq:alt_form} holds.
\end{proof}

\section{Completely join-irreducible elements and atoms}
\label{Sec:Irreducibles&Atoms}

Completely join-irreducible elements help in understanding the structure of a lattice. 
They are the building blocks in the sense that they cannot be broken down further within the lattice.
An element $j$ of a complete lattice $L$ is called 
\emph{completely join-irreducible} if $j = \bigvee S$
implies $j \in S$ for every subset $S$ of $L$ 
\cite[p.~242]{Davey02}. Note that the least element $0$ of $L$ is not completely
join-irreducible. The set of completely join-irreducible elements of $L$ is denoted by 
$\mathcal{J}(L)$, or simply by $\mathcal{J}$ if there is no danger of confusion. Note that if $L$ is finite, then $\mathcal{J}(L)$ consists of those elements which
cover precisely one element.

\begin{remark} \label{rem:join-irreducible_form}
For all $X\subseteq U$, $X^\Up = \bigcup \{ \{x\}^\Up \mid x \in X\}$.
Therefore, each completely join-irreducible element of $\wp(U)^\Up$ must be of the form $\{x\}^\Up = R(x)$ for some $x \in U$.
Similarly, every completely join-irreducible element of $\wp(U)^\UP$ must 
be of the form $\{x\}^\UP = \breve{R}(x)$, where $x \in U$.
Hence, the completely join-irreducible elements of $\wp(U)^\Up$ are elements $R(x)$ such that for all $\mathcal{H} \subseteq \wp(U)$, $R(x) = \bigcup \{ Y^\Up \mid Y \in \mathcal{H} \}$ implies that $R(x) = Y^\Up$ for some $Y \in \mathcal{H}$.
\end{remark}

Our following lemma characterizes the completely join-irreducible
elements of $\wp(U)^\Up$ in terms of $R$-neighbourhoods.

\begin{lemma} \label{lem:irreducible}
Let $R$ be a binary relation on $U$. The following are equivalent for $X \subseteq U$:
\begin{enumerate}[label={\rm (\roman*)}]
    \item $X^\Up$ is completely join-irreducible in 
    $\wp(U)^\Up$;
    \item $X^\Up = R(x)$ for some $x \in U$ such that
    for any $A \subseteq U$,
    $R(x) = \bigcup \{ R(a) \mid a \in A\}$ implies 
    $R(x) = R(b)$ for some  $b \in A$. 
\end{enumerate}
\end{lemma}

\begin{proof} (i)$\Rightarrow$(ii):
Suppose that $X^\Up$ is completely join-irreducible in
$\wp(U)^\Up$. By Remark~\ref{rem:join-irreducible_form},
$X^\Up = R(x)$ for some $x \in U$ such that
for all $\mathcal{H} \subseteq \wp(U)$, 
$R(x) = \bigcup \{ Y^\Up \mid Y \in \mathcal{H} \}$ implies that 
$R(x) = Y^\Up$ for some $Y \in \mathcal{H}$.
Let $R(x) = \bigcup \{R(a) \mid a \in A\}$ for some
$A \subseteq U$. Because each $R(a) = \{a\}^\Up$ belongs to
$\wp(U)^\Up$, we have $R(x) = R(b)$ for some $b \in A$. 

\medskip\noindent%
(ii)$\Rightarrow$(i): Let $X^\Up = R(x)$ be such that (ii) holds.
Assume that $R(x) = \bigcup \{ Y^\Up \mid Y \in \mathcal{H} \}$
for some $\mathcal{H} \subseteq \wp(U)$. Since
$\bigcup \mathcal{H} = \bigcup \{Y \mid Y \in \mathcal{H}\}$,
we have 
$\big (\bigcup \mathcal{H} )^\Up = 
\big ( \bigcup \{Y \mid Y \in \mathcal{H}\} \big )^\Up =
\bigcup \{Y^\Up \mid Y \in \mathcal{H}\}$. Hence, our assumption
is equivalent to $R(x) = \big ( \bigcup \mathcal{H} \big) ^\Up$.

It is easy to see that any set $X \subseteq U$ satisfies
$X^\Up =  \bigcup \{ \{a\}^\Up \mid a \in X\}$. By this,
\[ R(x) = \bigcup \{ \{a\}^\Up \mid a \in \bigcup \mathcal{H}\}
= \bigcup \{ R(a) \mid a \in \bigcup \mathcal{H}\}. \]
Because (ii) holds, we have that $R(x) = R(a)$ for some $a \in \bigcup \mathcal{H}$. 
Hence, there is $Y \in \mathcal{H}$ such that $a \in Y$. We get
\[ R(x) = R(a) = \{a\}^\Up \subseteq Y^\Up \subseteq 
\big ({\bigcup} \mathcal{H} \big ) ^\Up = R(x). \]
This means that $R(x) = Y^\Up$ and therefore $X^\Up = R(x)$ is
completely join-irreducible in $\wp(U)^\Up$. 
\end{proof}

The following corollary is obviously true, because $\breve{R}$ is reflexive whenever $R$ is. Therefore, this corollary follows directly from Lemma~\ref{lem:irreducible}.

\begin{corollary} \label{cor:irreducible}
Let $R$ be a binary relation on $U$. 
The following are equivalent for any $x \in U$:
\begin{enumerate}[label={\rm (\roman*)}]
    \item $X^\UP$ is completely join-irreducible in 
    $\wp(U)^\UP$;
    \item $X^\UP = \breve{R}(x)$ for some $x \in U$ such that
    for any $A \subseteq U$,
    $\breve{R}(x) = \bigcup \{ \breve{R}(a) \mid a \in A\}$ implies 
    $\breve{R}(x) = \breve{R}(b)$ for some  $b \in A$. 
\end{enumerate}
\end{corollary}

In the rest of this section, we assume that $R$ is a reflexive
binary relation on $U$. This is equivalent
to the fact that $x \in R(x)$. Note also that this means 
that  $x \in \mathcal{S}$ if and only if $R(x) = \{x\}$. 
In addition, if $x \in \mathcal{S}$, then 
$\{x\} \subseteq \{x\}^{\Up \DOWN} = R(x)^\DOWN = \{x\}^\DOWN
\subseteq \{x\}$. Thus, $R(x)^\DOWN =  \{x\}$ and, in particular,
$(R(x)^\DOWN, R(x)^\UP) = (\{x\}, \{x\}^\UP)$.
Our following lemma is technical in nature. We will need these properties in what follows.

\begin{lemma} \label{lem:singleton_join_irr}
Let $R$ be a reflexive relation on $U$ and 
$s \in \mathcal{S}$.
\begin{enumerate}[label={\rm (\roman*)}]
\item If $s \in \{b\}^\UP$, then $b = s$;
\item $\{s\}^\UP \in \mathcal{J}(\wp(U)^\UP)$;
\item $\{s\} = \{s\}^\Up \in \mathcal{J}(\wp(U)^\Up)$.
\end{enumerate}
\end{lemma}

\begin{proof}
(i) As $\{b\}^\UP = \breve{R}(b)$, $s \in \{b\}^\UP$ gives
$b \in R(s) = \{s\}$, that is, $b = s$.

(ii) Suppose $\{s\}^\UP = \bigcup_{a \in A} \{a\}^\UP$ for some
$A \subseteq U$. Because $s \in \{s\}^\UP$, $s \in \{a\}^\UP$ for some 
$a \in A$. By (i), $s = a$. Thus, $\{x\}^\UP = \{a\}^\UP$ and
$\{s\}^\UP$ is completely join-irreducible in $\wp(U)^\UP$
by Corollary~\ref{cor:irreducible},

(iii) Let $\{s\} = R(s) = \{s\}^\Up = \bigcup_{a \in A} \{a\}^\Up$ 
for some $A \subseteq U$. Because the size of each $\{a\}^\Up$ is at 
least one, there exists $a \in A$ such that $\{s\}^\Up = \{a\}^\Up$.
\end{proof}

Our next proposition shows that the completely join-irreducible
elements of  $\wp(U)^\Up$ and $\wp(U)^\UP$ determine completely 
join-irreducible elements of $\mathrm{DM(RS)}$.

\begin{proposition} \label{prop:join_irred}
Let $R$ be a reflexive relation on $U$ and $x \in U$.
\begin{enumerate}[label = {\rm (\roman*)}]
    \item If $x \notin \mathcal{S}$ and $\{x\}^\UP$ is completely join-irreducible in  $\wp(U)^\UP$, then 
    $(\{x\}^\DOWN, \{x\}^\UP)$ is completely join-irreducible in 
    $\mathrm{DM(RS)}$.
    
    \item If $\{x\}^\Up$ is is completely join-irreducible in $\wp(U)^\Up$, then $(\{x\}^{\Up \DOWN}, \{x\}^{\Up \UP})$ is completely join-irreducible in $\mathrm{DM(RS)}$.
\end{enumerate}    
\end{proposition}

\begin{proof}
(i) Because $x \notin \mathcal{S}$, $R(x) \nsubseteq \{x\}$ and so
$\{x\}^\DOWN = \emptyset$. Thus, 
$(\{x\}^\DOWN, \{x\}^\UP) = (\emptyset, \{x\}^\UP)$. Suppose 
that $(\emptyset, \{x\}^\UP) = 
\bigvee_{i \in I} (A_i,B_i)$ for some index set $I$ and
$(A_i,B_i)_{i \in I} \subseteq \mathrm{DM(RS)}$. Since
$\emptyset = \bigvee_{i \in I} A_i = (\bigcup_{i \in I} A_i)^{\Up \DOWN} 
\supseteq  \bigcup_{i \in I} A_i$, we have $A_i = \emptyset$ for all $i \in I$.

Now $\{x\}^\UP = \bigvee B_i$ gives that $\{x\}^\UP = B_k$ for some
$k \in I$ by Lemma~\ref{lem:irreducible}, because $\{x\}^\UP$ is completely join-irreducible in $\wp(U)^\UP$. 
Therefore, $(A_k,B_k) = (\emptyset, \{x\}^\UP)$ and 
$(\emptyset, \{x\}^\UP)$ is completely join-irreducible.

\medskip%

(ii) Let $\{x\}^\Up$ be completely join-irreducible in $\wp(U)^\Up$. 
Suppose that
\[ (\{x\}^{\Up \DOWN}, \{x\}^{\Up \UP}) = {\bigvee}_{i \in I}
(A_i,B_i) = \big ( ({\bigcup}_{i \in I} A_i )^{\Up\DOWN},
{\bigcup}_{i \in I} B_i \big ). \]
This means that 
\[
\{x\}^{\Up\DOWN} = ({\bigcup}_{i \in I} A_i )^{\Up\DOWN}.
\]
We obtain
\[
\{x\}^\Up = \{x\}^{\Up\DOWN \Up} = 
({\bigcup}_{i \in I} A_i )^{\Up\DOWN \Up} =
({\bigcup}_{i \in I} A_i )^{\Up} =
{\bigcup}_{i \in I} {A_i}^{\Up}.
\]
Because  $\{x\}^\Up$ is completely join-irreducible in $\wp(U)^\Up$, $\{x\}^\Up = {A_k}^\Up$ for some $k \in I$. This implies
$\{x\}^{\Up \DOWN} = {A_k}^{\Up\DOWN} = A_k$. Note that
$A_k \in \wp(A)^\DOWN$.

The fact ${A_k}^{\Up\UP} \subseteq B_k$ implies
\[ \{x\}^{\Up\UP} = {A_k}^{\Up\UP} \subseteq B_k .\]
On the other hand, $\{x\}^{\Up\UP} = {\bigcup}_{i \in I} B_i \supseteq B_k$.
Thus, $\{x\}^{\Up\UP} = B_k$. We have proved that 
$(\{x\}^{\Up \DOWN}, \{x\}^{\Up \UP})
= (A_k,B_k)$, that is,  $(\{x\}^{\Up \DOWN}, \{x\}^{\Up \UP})$ is completely 
join-irreducible.
\end{proof}

We can now write the following characterization of complete join-irreducible
elements.

\begin{theorem} \label{thm:join_irreducibles}
Let $R$ be a reflexive relation on $U$. The set of completely join-irreducible
elements of $\mathrm{DM(RS)}$ is
\begin{gather} \label{eq:Join1}
\{ (\{x\}^{\Up \DOWN},\{x\}^{\Up \UP}) \mid 
\text{$\{x\}^\Up$ is completely join-irreducible in $\wp(U)^\Up$} \} \\
\label{eq:Join2}
\cup \  \{ (\{x\}^\DOWN,\{x\}^\UP) \mid 
\text{$\{x\}^\UP$ is completely join-irreducible in $\wp(U)^\UP$ 
and $x \notin \mathcal{S}$} \}. 
\end{gather}
\end{theorem}

\begin{proof} In view of Proposition~\ref{prop:join_irred}, 
elements in sets \eqref{eq:Join1} and \eqref{eq:Join2} are 
completely join-irreducible in $\mathrm{DM(RS)}$.

Conversely, assume now that $(A,B)$ is a completely join-irreducible 
element in $\mathrm{DM(RS)}$. Then, in view of Lemma~\ref{lem:rule}, 
(i) $(A,B) = (\{a\}^{\Up \DOWN}, \{a\}^{\Up \UP})$ for some $a \in A$ or
(ii) $(A,B) =  (\{b\}^\DOWN,\{b\}^\UP)$, where
$b \in B^\Down \subseteq B$.

(i) Clearly, in the case $(A,B) = (\{a\}^{\Up \DOWN}, \{a\}^{\Up \UP})$
we have to prove only that $\{a\}^\Up$ is completely join-irreducible
in $\wp(U)^\Up$. This can be done by using the equivalence of
Lemma~\ref{lem:irreducible}. 
Assume that $R(a) = \bigcup \{R(c)\mid c\in C\}$ for some $C \subseteq U$.
This gives that 
\begin{align*}
\{a\}^{\Up \DOWN} = R(a)^\DOWN 
&= \big (\bigcup \{R(c) \mid c \in C\} \big )^\DOWN 
= \big (\bigcup \{\{c\}^\Up \mid c \in C\} \big )^\DOWN \\
&= \big (\bigcup \{\{c\}^{\Up \DOWN \Up} \mid c \in C\} \big )^\DOWN 
= \big (\bigcup \{\{c\}^{\Up \DOWN} \mid c \in C\} \big )^{\Up \DOWN}.
\end{align*}
Similarly,
\begin{align*}
\{a\}^{\Up \UP} = R(a)^\UP 
&= \big (\bigcup \{R(c) \mid c \in C\} \big )^\UP 
= \bigcup \{R(c)^\UP \mid c \in C\}  \\
&= \bigcup \{ \{c\}^{\Up\UP} \mid c \in C\}  \\
\end{align*}
This means that 
\begin{align*}
 (\{a\}^{\Up \DOWN}, \{a\}^{\Up \UP}) 
 & =  
 \Big (\big (\bigcup \{\{c\}^{\Up \DOWN} \mid c \in C\} \big )^{\Up \DOWN},
 \bigcup \{ \{c\}^{\Up\UP} \mid c \in C\} \Big ) \\
 &=  \bigvee\big \{(\{c\}^{\Up \DOWN}, \{c\}^{\Up \UP })\mid c\in C \big \}.
\end{align*}
Since $(\{a\}^{\Up \DOWN}, \{a\}^{\Up \UP})$ is
completely join-irreducible,
$(\{a\}^{\Up \DOWN}, \{a\}^{\Up \UP}) =
 (\{c\}^{\Up \DOWN}, \{c\}^{\Up \UP})$ for some $c\in C$. 
Therefore, $\{a\}^{\Up \DOWN} = \{c\}^{\Up \DOWN}$ and 
$\{a\}^\Up =\{a\}^{\Up \DOWN \Up} =
 \{c\}^{\Up \DOWN \Up} = \{c\}^\Up$.
This proves that $\{a\}^\Up$ is completely join-irreducible
in $\wp(U)^\Up$.

(ii) Suppose that $(A,B) =  (\{b\}^\DOWN,\{b\}^\UP)$, where
$b \in B^\Down$. If $b \in \mathcal{S}$, then
$\{b\}^\Up = R(b) = \{b\}$ gives that
$(\{b\}^\DOWN,\{b\}^\UP) =  (\{b\}^{\Up \DOWN}, \{b\}^{\Up \UP})$.
So, this case is  already covered by part (i).

If $b \notin \mathcal{S}$, then
$(A,B) = (\{b\}^\DOWN,\{b\}^\UP) = (\emptyset,\{b\}^\UP)$.
We show that $\{b\}^\UP$ is completely join-irreducible in
$\wp(U)^\UP$.  Suppose that 
$\{b\}^\UP = \breve{R}(b)=\ \bigcup \{\breve{R}(d)\mid d\in D\}$ 
for some $D\subseteq U$. Since $\breve{R}$ is reflexive, we
have that $d\in \breve{R}(d) \subseteq \breve{R}(b)$ for all $d\in D$. 
Thus, $b\in R(d)$ for each $d\in D$. Observe that this yields 
$d \notin \mathcal{S}$  for every $d\in D$.
Indeed, $d\in \mathcal{S}$ and $b \in R(d)$ would imply
$b = d \in \mathcal{S}$, contradicting $b \notin \mathcal{S}$.
Therefore, $\{d\}^\DOWN = \emptyset$ for all $d\in D$. 
%Note also that $\bigcup \{ \{d\}^\UP  \mid d \in D\}$ belongs
%to $\wp(U)^\UP$ and that 
%$(\emptyset, \bigcup \{ \{d\}^\UP  \mid d \in D\})$
%is in $\mathrm{DM(RS)}$. 

We can now write 
\[ {\bigvee}_{d \in D} (\{d\}^\DOWN, \{d\}^\UP) =
\Big  ( \big ( \bigcup \{d\}^\DOWN \big )^{\Up\DOWN},
\bigcup \{d\}^\UP \Big )=
\Big ( \emptyset^{\Up\DOWN}, \bigcup \{d\}^\UP \Big ) = 
\Big (\emptyset, \bigcup \breve{R}(d) \Big ).
\]
Since $\bigcup \breve{R}(d) = \breve{R}(b) = \{b\}^\UP$, we have
\[ {\bigvee}_{d \in D} (\{d\}^\DOWN, \{d\}^\UP) =
(\emptyset, \{b\}^\UP ).\]
Because $(\emptyset ,\{b\}^\UP)$ is completely
join-irreducible in $\mathrm{DM(RS)}$, we get 
$(\emptyset, \{b\}^\UP) =
(\{d\}^\DOWN,\{d\}^\UP) = 
(\emptyset,\{d\}^\UP)$ for some $d\in D$. 
This yields $\breve{R}(b) = \{b\}^\UP = \{d\}^\UP = \breve{R}(d)$.
By Corollary~\ref{cor:irreducible},
$\{b\}^\UP$ is completely join-irreducible in $\wp(U)^\UP$.
\end{proof}

We will now present two examples. It is well-known that rough set lattice $\mathrm{RS}$ determined by an equivalence is a regular double Stone algebra \cite{Come93}. By
Example~\ref{exa:JoinRS} we can induce similar
structures also by relations which are merely reflexive.
Example~\ref{exa:no_join_irreducibles} shows that there
are such reflexive relations that $\wp(U)^\Up$ has no
join-irreducbiles.

\begin{example} \label{exa:JoinRS}
Let $R$ be a reflexive relation on $U = \{1,2,3\}$
such that 
\[
R(1) = \{1\}^\Up = \{1,2,3\}, \quad 
R(2) = \{2\}^\Up = \{2\}, \quad 
R(3) = \{3\}^\Up = \{1,3\}. 
\]
Note that the relation $R$ is not symmetric, because $(1,2) \in R$, but $(2,1) \notin R$. Similarly, $R$ is not transitive, because $(3,1) \in R$ and $(1,2)\in R$,
but $(3,2) \notin R$.

If $R$ is considered a directional similarity relation, then
element $1$ is similar to $2$ and $3$, and element $3$ is 
similar to $1$. Element $2$ is not similar to any other element, 
so it is a singleton and $\mathcal{S} = \{2\}$.

The inverse $\breve{R}$ of $R$ can be interpreted to that 
$\breve{R}(x)$ consists of elements which are similar to $x$.
We can now form the sets:
\[ \breve{R}(1) = \{1\}^\UP = \{1,3\}, \quad
\breve{R}(2) = \{2\}^\UP = \{1,2\}, \quad
\breve{R}(3) = \{3\}^\UP = \{1,3\}.\] 
We have the complete lattices
 \[
 \wp(U)^\UP = \{\emptyset, \{1,2\}, \{1,3\}, U\} 
 \quad 
 \text{ and } 
 \quad
 \wp(U)^\Up = \{\emptyset, \{2\}, \{1,3\}, U\}.
 \]
The completely join-irreducible elements of these lattices are
\[ \mathcal{J}(\wp(U)^\UP) = \{ \{1,2\}, \{1,3\}\} 
 \quad 
 \text{ and } 
 \quad
\mathcal{J}(\wp(U)^\Up) = \{  \{2\}, \{1,3\}\}.\]

According to Theorem~\ref{thm:join_irreducibles}, 
the join-irreducibles of $\mathrm{DM(RS)}$ are
\begin{gather*}
\{ (\{x\}^{\Up \DOWN},\{x\}^{\Up \UP}) \mid 
\{x\}^\Up \in \mathcal{J}(\wp(U)^\Up) \}  =
\{ (\{2\}^\DOWN, \{2\}^\UP),  (\{1,3\}^\DOWN, \{1,3\}^\UP) \} \\
 = \{ (\{2\}, \{1,2\}),  (\{3\}, \{1,3\}) \}.
\end{gather*} 
together with
\[
\{ (\emptyset,\{x\}^\UP) \mid \text{$\{x\}^\UP \in 
\mathcal{J}(\wp(U)^\UP)$ and $x \notin \mathcal{S}$} \} 
= \{ (\emptyset, \{1,3\}) \}.
\]

The complete lattice $\mathrm{DM(RS)}$ is given in 
Figure~\ref{fig1:completion1}. 
It forms a well-known regular double Stone algebra isomorphic to
$\mathbf{2} \times \mathbf{3}$.
It can be easily verified that in this example,
$\mathrm{RS}$ is equal to $\mathrm{DM(RS)}$.
\begin{figure}[h]
\includegraphics[width=70mm]{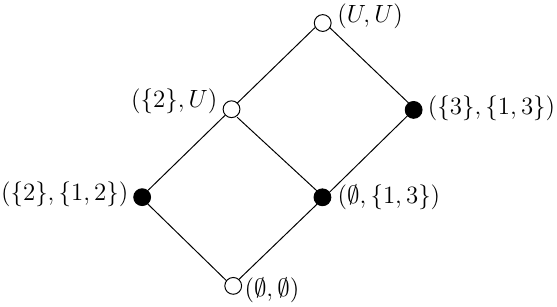}
\caption{
The complete lattice $\mathrm{DM(RS)}$ of Example~\ref{exa:JoinRS}, forming a regular double Stone algebra isomorphic to $\mathbf{2} \times \mathbf{3}$.}
\label{fig1:completion1}
\end{figure}
\end{example}

\begin{example} \label{exa:no_join_irreducibles}
Any binary relation $R$ on $U$ is reflexive if and only if $x \in R(x)$ for 
all  $x \in U$.
We set $U = \mathbb{N} = \{1,2,3,\ldots\}$. Let us
define a relation $R$ on $U$ such that the set of
completely join-irreducible elements of $\wp(U)^\Up$ will be empty. 
As we already noted, every completely join-irreducible element of 
$\wp(U)^\Up$ must be of the form $R(x)$ for some $x \in U$.

Let us denote by $[a]_n$ the \emph{congruence class} of $a \in \mathbb{N}$ 
modulo $n$, that is, 
\[ [a]_n = \{x \in \mathbb{N} \mid x \equiv a \pmod n\} .\] 
We define the sets $R(n)$ in terms of the congruence classes modulo
$2^{\lfloor \lg n \rfloor}$. Here $\lfloor x \rfloor$ denotes the least integer smaller or equal to $x$ and $\lg x$ is the 2-base logarithm of $x$.

For any $n \in \mathbb{N}$,
\[ R(n) = [n]_{k}, \text{ where $k = 2^{\lfloor \lg n \rfloor}$}.\]

\begin{itemize}
\item If $\lfloor \lg n \rfloor = 0$, then $k = 2^0 = 1$ and
\begin{gather*}
R(1) = [1]_1 = \mathbb{N}.  
\end{gather*}

\item If $\lfloor \lg n \rfloor = 1$, then $k = 2^1 = 2$ and
\begin{gather*}
    R(2) = [2]_2 = \{2,4,6,8, 10\ldots\}, \quad
    R(3) = [3]_2 = \{1,3,5,7,9\ldots\}.
\end{gather*}

\item If $\lfloor \lg n \rfloor = 2$, then $k = 2^2 = 4$ and
\begin{gather*}
R(4) = [4]_4 = \{4,8,12,16,20\ldots\}, \quad
R(5) = [5]_4 = \{1,5,9,13,17,\ldots\},\\
R(6) = [6]_4 = \{2,6,10,14,18,\ldots\}, \quad 
R(7) = [4]_4 = \{3,7,11,15,19,\ldots\}.
\end{gather*}

\item If $\lfloor \lg n \rfloor = 3$, then $k = 2^3 = 8$ and
\[ R(n) =  [n]_8  \text{ for $8 \leq n \leq 15$.}\]

\item If $\lfloor \lg n \rfloor = 4$, then $k = 2^3 = 16$ and
\[ R(n) =  [n]_{16}  \text{ for $16 \leq n \leq 31$.}\]
\end{itemize}

The idea is that each $R(n)$ is divided into two by taking every second element. If $n$ is even, $R(n)$ is split into  $R(2n)$ and $R(2n+2)$. If $n$
is odd, then $R(n)$ will we split into $R(2n-1)$ and $R(2n+1)$.
Because each $R(n)$ can be always partitioned into two, there
cannot be join-irreducible elements in $\wp(U)^\Up$. 
\end{example}

Let $L$ be a lattice with the least element $0$. An element 
$a$ is an \emph{atom} of $L$ if it covers $0$, that is, $0 \prec a$ \cite[p.~113]{Davey02}. Our final result in this section describes the atoms of $\mathrm{DM(RS)}$.

\begin{proposition} \label{prop:atoms}
If $R$ is a reflexive relation on $U$,
the set of the atoms of $\mathrm{DM(RS)}$ is 
\begin{equation} \label{eq:atoms}
\{ (\{x\}^\DOWN,\{x\}^\UP ) \mid 
\text{$\{x\}^\UP$ is an atom in $\wp(U)^\UP$} \}.
\end{equation}
\end{proposition}

\begin{proof}
First, we prove that all elements of \eqref{eq:atoms} are atoms 
in $\mathrm{DM(RS)}$. 
Assume that $(A,B) \in \mathrm{DM(RS)}$ and
$(\emptyset,\emptyset) < (A,B) \leq (\{x\}^\DOWN,\{x\}^\UP)$.
First note that $B \neq \emptyset$, because $A \subseteq B$.
Let $x$ be such that
$\{x\}^\UP$ is an atom in $\wp(U)^\UP$. 
Now $\emptyset < B \leq \{x\}^\UP$ gives
$B = \{x\}^\UP$ as $\{x\}^\UP$ is an atom of
$\wp(U)^\UP$.
If $x \notin \mathcal{S}$, then $A \subseteq 
\{x\}^\DOWN = \emptyset$
gives $A = \emptyset$. Hence, $(A,B) = (\emptyset,\{x\}^\UP)$,
which proves that $(\{x\}^\DOWN,\{x\}^\UP) = 
(\emptyset,\{x\}^{\blacktriangle})$ is an atom of
$\mathrm{DM(RS)}$.
If $x\in \mathcal{S}$, then 
$(\{x\}^\DOWN,\{x\}^\UP) =  (\{x\},\{x\}^\UP)$.
Now, by the definition of 
$\mathrm{DM(RS)}$, $x\in B\cap \mathcal{S} =
A \cap \mathcal{S}$. Thus,
$x \in A \subseteq \{x\}$ and $A = \{x\}$.
This proves that
$(\{x\}^\DOWN,\{x\}^\UP) =  (\{x\},\{x\}^\UP)$ is an atom of 
$\mathrm{DM(RS)}$.

\smallskip%
Conversely, suppose that $(A,B)$ is an atom of $\mathrm{DM(RS)}$. 
We are going to show that $(A,B)$ belongs to the set
\eqref{eq:atoms}. Observe first 
that for any $x\in U$ and $\emptyset \neq C\in \wp(U)^\UP$,  
$C \subset \{x\}^\UP$ yields that there exists $z\notin \mathcal{S}$ 
with $\{z\}^\UP \subset \{x\}^\UP$. 
Indeed, since $C = Z^\UP$ for some $\emptyset\neq Z\subseteq U$,
there is $z\in Z$ with 
$\{z\}^\UP \subseteq Z^\UP = C \subseteq \{x\}^\UP$. 
Now, $z\in \{z\}^\UP \subset \{x\}^\UP = \breve{R}(x)$ implies
$x\in R(z)$. If we assume that $z \in \mathcal{S}$, then 
$x\in R(z)=\{z\}$ gives $x=z$ and $\{x\}^\UP = \{z\}^\UP$, a contradiction. Hence, $z\notin \mathcal{S}$.

\smallskip%

Since atoms are  completely join-irreducible elements, according to Theorem~\ref{thm:join_irreducibles}, $(A,B)$ must be of the form
\begin{equation}\tag{A}
(\emptyset,\{x\}^\UP)=(\{x\}^\DOWN,\{x\}^\UP), 
\text{ where $x \notin \mathcal{S}$
and $\{x\}^\UP \in \mathcal{J}(\wp(U)^\UP)$}
\end{equation}
or
\begin{equation} \tag{B}
(\{x\}^{\Up\DOWN},\{x\}^{\Up\UP}), 
 \text{where $\{x\}^\Up \in\mathcal{J}(\wp(U)^\Up)$}.
\end{equation}

In case (A), take any $C \in \wp(U)^\UP$ with 
$\emptyset \subset C \subseteq \{x\}^\UP$.
If $C \neq\{x\}^\UP$, then, by the above, there exists
$z \notin \mathcal{S}$ such that 
$\{z\}^\UP \subset \{x\}^\UP$. Now, 
$(\emptyset,\{z\}^\UP) =(\{z\}^\DOWN,\{z\}^\UP)$ belongs to
$ \mathrm{DM(RS)}$ and
$(\emptyset,\emptyset) \subset (\emptyset,\{z\}^\UP) < 
(\emptyset, \{x\}^\UP) = (A,B)$. This is a contradiction, 
because $(A,B)$ is an atom of $\mathrm{DM(RS)}$. Thus, 
we must have that $C=\{x\}^\UP$, which proves that
$\{x\}^\UP$ is an atom of $\wp(U)^\UP$.

In case (B), observe that $x\in \mathcal{S}$ is necessary. 
Indeed, if $x\notin \mathcal{S}$, then 
$\{x\}\subseteq\{x\}^\Up$ and $\{x\}^\DOWN = \emptyset$ would imply $(\emptyset,\emptyset) \neq 
(\{x\}^\DOWN,\{x\}^\UP) = (\emptyset,\{x\}^\UP)<
(\{x\}^{\Up\DOWN},\{x\}^{\Up\UP})$. 
As $(A,B) = (\{x\}^{\Up\DOWN},\{x\}^{\Up\UP})$
is an atom of $\mathrm{DM(RS)}$, we have a contradiction. 
Hence, $x\in S$ and
$(A,B)=(\{x\}^{\Up\DOWN},\{x\}^{\Up\UP})=(\{x\},\{x\}^\UP)$. 

Let  $C \in\wp(U)^\UP$ be such that
$\emptyset \subset C\subseteq\{x\}^\UP$. 
If $C\neq\{x\}^\UP$, then there is 
$z \notin \mathcal{S}$ with
$\{z\}^\UP \subset\{x\}^\UP$, as we noted above. 
%%% here
Now $z \notin \mathcal{S}$ yields 
$(\{z\}^\DOWN,\{z\}^\UP)=(\emptyset,\{z\}^\UP) <
(\{x\},\{x\}^\UP)$. This contradicts with the assumption
that $(A,B) = (\{x\},\{x\}^\UP)$ is an atom of $\mathrm{DM(RS)}$.
Thus, we get $C=\{x\}^\UP$, proving that 
$\{x\}^\UP$ is an atom of $\wp(U)^\UP$.
\end{proof}

A lattice $L$ with $0$ is called \emph{atomic}, if for any 
$x \neq 0$ there exists an atom $a \leq x$ \cite[p.~113]{Davey02}. Because there
can be reflexive relations $R$ on $U$ such that the set of 
completely join-irreducible elements of $\wp(U)^\UP$ is empty, 
the set of the atoms of $\mathrm{DM(RS)}$ can also be empty
in view of Proposition~\ref{prop:atoms}. So,
$\mathrm{DM(RS)}$ is not necessarily atomic.

\section{On complete distributivity}
\label{Sec:CompleteDistributivitity}

The idea behind complete distributivity is that it extends the usual distributive property of lattices to arbitrary joins and meets. In a completely distributive lattice, the order of taking joins and meets can be interchanged in a very general way, which is not necessarily true in ordinary distributive lattices.
A complete lattice $L$ is \emph{completely distributive} 
if for any doubly indexed
subset $\{x_{i,\,j}\}_{i \in I, \, j \in J}$ of $L$, we have
\[
\bigwedge_{i \in I} \Big ( \bigvee_{j \in J} x_{i,\,j} \Big ) = 
\bigvee_{ f \colon I \to J} \Big ( \bigwedge_{i \in I} x_{i, \, f(i) } \Big ), \]
that is, any meet of joins may be converted into the join of all
possible elements obtained by taking the meet over $i \in I$ of
elements $x_{i,\,k}$\/, where $k$ depends on $i$
\cite[p.~239]{Davey02}.

Our next proposition shows that the complete distributivity of the completion $\mathrm{DM(RS)}$ is tightly connected with the complete distributivity of the approximation lattices.

\begin{proposition} \label{prop:completely_distributive}
Let $R$ be any binary relation on $U$. The following are equivalent:
\begin{enumerate}[label={\rm (\roman*)}]
    \item $\mathrm{DM(RS)}$ is completely distributive;
    \item Any of $\wp(U)^\UP$, $\wp(U)^\DOWN$,  $\wp(U)^\Up$, $\wp(U)^\Down$ 
    is completely distributive.  
\end{enumerate}
\end{proposition}

\begin{proof} (i)$\Rightarrow$(ii): 
Suppose $\mathrm{DM(RS)}$ is completely distributive.
Since $\wp(U)^\DOWN$ is the image of $\mathrm{DM(RS)}$ under the complete lattice-homomorphism $\pi_1$, it is completely distributive.
Because of the isomorphisms \eqref{eq:isomorphisms}, so are $\wp(U)^\UP$, $\wp(U)^\Up$, $\wp(U)^\Down$ 

\medskip\noindent%
(ii)$\Rightarrow$(i): Suppose that (ii) holds. Then, 
$\wp(U)^\DOWN$ and $\wp(U)^\UP$ are completely distributive
because of the isomorphisms \eqref{eq:isomorphisms}.
The direct product of completely
distributive lattices is completely distributive. Therefore,
$\wp(U)^\DOWN \times \wp(U)^\UP$ is completely distributive. 
By Proposition~\ref{prop:uma_operations} $\mathrm{DM(RS)}$ is a complete sublattice of  $\wp(U)^\DOWN \times \wp(U)^\UP$.
Thus, it is completely distributive.
\end{proof}

A complete lattice $L$ is \emph{spatial} if each element
is a join of completely join-irreducible elements, that is,
for all $x \in L$,
\[ x = \bigvee \{ j \in \mathcal{J} \mid j \leq x \}. \]
It is well known that every finite lattice is spatial.
Note that using the terminology of \cite{Davey02}, spatial lattices are complete lattices in which the set $\mathcal{J}$ is join-dense.

\begin{remark}\label{rem:alex}
(a) An \emph{Alexandroff topology} is a topology which is 
closed under arbitrary intersections.
They were first introduced in 1937 by P.~S.~Alexandroff under the name ``discrete spaces'' 
in \cite{Alex37}.

If $\mathcal{T}$
is an Alexandroff topology on $U$, then $(\mathcal{T},\subseteq)$
forms a complete lattice such that
$\bigvee \mathcal{H} = \bigcup \mathcal{H}$ and 
$\bigwedge \mathcal{H} = \bigcap \mathcal{H}$ for
$\mathcal{H} \subseteq \mathcal{T}$. The least and
the greatest elements are $\emptyset$ and $U$. This
lattice is known to be completely distributive.

For any $x \in U$, the set 
\[ N_\mathcal{T}(x) = \bigcap \{X \in \mathcal{T} \mid x \in X\}\] 
is the smallest set in $\mathcal{T}$ containing $x$. The set
$N_\mathcal{T}(x)$ is called the \emph{neighbourhood} of $x$.
The set of the completely join-irreducibles of $\mathcal{T}$
is $\{N_\mathcal{T}(x) \mid x \in U\}$. The lattice
$\mathcal{T}$ is known to be spatial.

(b) Let $R$ be a binary relation on $U$.
Since for any $X \subseteq U$, 
$X^\UP = \bigcup \{ \{x\}^\UP \mid x \in X\}$, 
we have that $\wp(U)^\UP$ is spatial if and only if each $\{x\}^\UP$ is a
join of some completely join-irreducible elements of $\wp(U)^\UP$.
Similarly, $\wp(U)^\Up$ is spatial if and only if each $\{x\}^\Up$ 
is a join of some completely join-irreducible elements of 
$\wp(U)^\Up$.
\end{remark}

\begin{example} \label{Exa:ThreeExamples}
(a) As we show in Example~\ref{exa:no_join_irreducibles},
the set of completely join-irreducibles of $\wp(U)^\Up$
can be empty. In such a case, $\wp(U)^\Up$ cannot be
spatial.

\smallskip\noindent%
(b) It is known that $\wp(U)^\DOWN$, $\wp(U)^\UP$, $\wp(U)^\Down$, $\wp(U)^\Up$ form Alexandroff topologies when $R$ is a quasiorder. In addition, 
$\wp(U)^\DOWN = \wp(U)^\Up$ and $\wp(U)^\Down = \wp(U)^\UP$.
\end{example}

Our next proposition shows that the spatiality of $\mathrm{DM(RS)}$ is equivalent to both $\wp(U)^\UP$ and $\wp(U)^\Up$ being spatial.

\begin{proposition} \label{prop:spatial}
Let $R$ be a reflexive relation on $U$. Then, the following assertions are equivalent:
\begin{enumerate}[label={\rm (\roman*)}]
\item $\mathrm{DM(RS)}$ is spatial;
\item $\wp(U)^\UP$ and $\wp(U)^\Up$ are spatial.
\end{enumerate}
\end{proposition}

\begin{proof}
(i)$\Rightarrow$(ii). Assume that $\mathrm{DM(RS)}$ is spatial.
To prove that $\wp(U)^\UP$ is spacial it is enough to
show that for any $y\in U$, $\{y\}^\UP$ is a union of
some completely join-irreducible elements of $\wp(U)^\UP$, as noted in
Remark~\ref{rem:alex}.

In view of Lemma~\ref{lem:singleton_join_irr}(ii), this obviously holds for any $y\in \mathcal{S}$.  Let $y \notin \mathcal{S}$.
Since now $(\emptyset,\{y\}^\UP)=(\{y\}^\DOWN,\{y\}^\UP)
\in \mathrm{RS} \subseteq \mathrm{DM(RS)}$, (i) means that 
$(\emptyset,\{y\}^\UP)$ is the join of some  
$(A_i,B_i)_{i \in I}\subseteq \mathcal{J}(\mathrm{DM(RS)})$,
that is,
\begin{equation} \label{eq:join_nonsingleton}
(\emptyset,\{y\}^\UP)= \Big ( \big (\bigcup
\{A_i\}_{i \in I}\big)  ^{\Up\DOWN},
\bigcup \{B_i\}_{i \in I}\Big )  
\end{equation}
Then $A_i\subseteq {A_i}^{\Up\DOWN} \subseteq 
( \bigcup \{A_i\}_{i \in I})^{\Up\DOWN}=\emptyset$ yields
$A_i=\emptyset$ for all $i\in I$.
In view of Theorem~\ref{thm:join_irreducibles}, every
completely join-irreducible element $(A_i,B_i)=(\emptyset,B_i)$
of $\mathrm{DM(RS)}$ equals $(\emptyset,\{z_i\}^\UP)$, where 
$z_i \notin \mathcal{S}$ and $\{z_i\}^\UP$ is completely join-irreducible in $\wp(U)^\UP$. Then, \eqref{eq:join_nonsingleton}
implies $\{y\}^\UP = \bigcup_{i \in I} \{z_i\}^\UP$. 
This proves that $\wp(U)^\UP$ is spatial.

In order to prove that $\wp(U)^\Up$ is spatial, it is enough to show 
that any $\{x\}^\Up$ is a union of some completely join-irreducible elements of $\wp(U)^\Up$. 
Clearly, $(\{x\}^{\Up\DOWN},\{x\}^{\Up\UP}) \in
\mathrm{RS} \subseteq \mathrm{DM(RS)}$. Hence, there are sets
$Y,Z\subseteq U$ such that
\begin{equation}\label{eq:join_second}
(\{x\}^{\Up\DOWN},\{x\}^{\Up\UP}) =
\big( \bigvee 
\{(\{y\}^{\Up\DOWN},\{y\}^{\Up\UP}) \mid y\in Y\}\big)  
\vee 
\big (\bigvee \{(\emptyset,\{z\}^\UP) \mid z\in Z\}\big )
\end{equation}
with $\{y\}^\Up \in \mathcal{J}(\wp(U)^\Up)$ for any $y\in Y$,
$\{z\}^\UP \in \mathcal{J}(\wp(U)^\UP)$, and 
$z\notin \mathcal{S}$ for
each $z\in Z$. Then, 
$\{x\}^{\Up\DOWN} =\left( \bigcup
\{\{y\}^{\Up\DOWN} \mid y\in U\}\right)^{\Up\DOWN}$ implies
\begin{gather*}
\{x\}^{\Up} = \{x\}^{\Up\DOWN\Up} =
\big ( \bigcup \{\{y\}^{\Up\DOWN} \mid y\in U\} \big )^{\Up\DOWN\Up} = 
\big ( \bigcup \{\{y\}^{\Up\DOWN} \mid y\in U\} \big )^{\Up} = \\
\bigcup \{\{y\}^{\Up\DOWN\Up} \mid y\in U\} =
\bigcup \{\{y\}^{\Up} \mid y\in U\}.
\end{gather*}
This means that also $\wp(U)^\Up$ is spatial.

(ii)$\Rightarrow$(i). Assume that (ii) holds. Then, in view of 
Lemma~\ref{lem:rule}, to prove that $\mathrm{DM(RS)}$ is spatial, it is enough to show that the following elements are the joins of 
some completely join-irreducible elements of $\mathrm{DM(RS)}$:
\begin{enumerate}[label={\rm (\alph*)}]
\item elements $(\{a\}^{\Up\DOWN},\{a\}^{\Up\UP})$ for any $a\in U$;

\item elements $(\{b\}^{\DOWN},\{b\}^{\UP})$
for any $b\in U$.
\end{enumerate}

Case (a): Since $\wp(U)^\Up$ is spatial, there
exists a set $X \subseteq U$ such that for all $x \in X$
$\{x\}^\Up \in \mathcal{J}(\wp(U)^\Up)$ and 
$\{a\}^\Up = \bigcup \{\{x\}^\Up \mid x \in X\}$. 
Now $(\{x\}^{\Up\DOWN},\{x\}^{\Down\UP})$ belongs to 
$\mathcal{J}(\mathrm{DM(RS)})$ for each $x \in X$. Furthermore,
\[ \{a\}^{\Up\UP}=\bigcup \{\{x\}^{\Up\UP}\mid x \in X\},\] 
and 
\[\{a\}^{\Up \DOWN} = (\bigcup
\{\{x\}^\Up \mid x \in X \})^{\DOWN} =
(\bigcup \{\{x\}^{\Up\DOWN\Up} \mid x \in X\} )^\DOWN =
( \bigcup \{\{x\}^{\Up\DOWN}\mid x\in X\})^{\Up\DOWN}.\]
Hence, we obtain:
\begin{align*}
(\{a\}^{\Up\DOWN},\{a\}^{\Up\UP})& = 
(( \bigcup
\{\{x\}^{\Up\DOWN}\mid x \in X\} )^{\Up\DOWN},
\bigcup
\{\{x\}^{\Up\UP} \mid x \in X\}) \\
&=\bigvee
%EndExpansion
\{ (\{x\}^{\Up\DOWN},\{x\}^{\Up\DOWN}) \mid x \in X\}.
\end{align*}

Case (b): If $b\in \mathcal{S}$, then $\{b\}^\Up=R(b)=\{b\}$. Hence,
$(\{b\}^\DOWN,\{b\}^\UP) =(\{b\}^{\Up\DOWN},\{b\}^{\Up\UP})$ and this subcase is covered by (a). Suppose that $b\notin \mathcal{S}$. Then,
$(\{b\}^\DOWN,\{b\}^\UP)=(\emptyset,\{b\}^\UP)$. In view of our assumption, $\{b\}^\UP$ is a union of some completely join-irreducible elements of 
$\wp(U)^\UP$, that is,
$\{b\}^\UP = \bigcup \{\{y\}^\UP \mid y\in Y\}$ for some $Y \subseteq U$.
Observe that any $y \notin \mathcal{S}$. Indeed,
for each $y\in Y$, $y\in\{y\}^\UP \subseteq \{b\}^\UP$. In view of 
Lemma~\ref{lem:singleton_join_irr}(i), $y\in \mathcal{S}$ would imply $b=y \in \mathcal{S}$, a contradiction to our hypothesis. Thus, 
$(\emptyset,\{y\}^\UP)=(\{y\}^\DOWN,\{y\}^\UP)\in 
\mathrm{RS} \subseteq \mathrm{DM(RS)}$ for each $y\in Y$, and
\[ (\{b\}^\DOWN,\{b\}^\UP)=(\emptyset,\{b\}^\UP)=
\bigvee \{(\emptyset,\{y\}^\UP)\mid y\in Y\}.\]
By Theorem~\ref{thm:join_irreducibles}, for each $y \in Y$, the
pair $(\emptyset,\{y\}^\UP)$ is a completely join-irreducible element of 
$\mathrm{DM(RS)}$. So, our proof is completed.
\end{proof}

Note that since the pair $({^\UP},{^\Down})$ forms a
Galois connection on $(\wp(U),\subseteq)$, the
complete lattices $(\wp(U)^\UP,\subseteq)$ and $(\wp(U)^\Down,\subseteq)$
are isomorphic by the map  $X^\UP \mapsto X^{\UP\Down}$. Therefore,
$\wp(U)^\UP$ is spatial if and only if $\wp(U)^\Down$ is spatial.
 Similar observation holds between
$\wp(U)^\Up$ and $(\wp(U)^\DOWN$.

An element $p$ of a complete lattice $L$ is said to be 
\emph{completely join-prime} if for every $X \subseteq L$, $p \leq \bigvee X$ implies $p \leq x$
for some $x \in X$ \cite[p.~242]{Davey02}
Note that in a complete lattice $L$, each completely join-prime element is completely join-irreducible. This
can be seen by assuming that $p$ is completely join-prime.
If $p = \bigvee S$ for some $S \subseteq L$, then
$p \leq \bigvee S$ gives that $p \leq x$ for some $x \in S$.
On the other hand, $p \geq s$ for all $s \in S$. Therefore,
$p = x$ and $x$ is join-irreducible. The converse does not
always hold, as we can see in Example~\ref{exa:completely_join_prime}.

However, if $L$ is a completely distributive lattice,
then the set of completely join-prime and completely join-irreducible elements coincide \cite{Balachandran55}.
Note also that $0$ is not completely join-prime.
We denote by $\mathcal{J}_p(L)$ the set of all completely
join-prime elements of $L$.
G.~N.~Raney \cite[Theorem~2]{Raney52} has stated that a 
complete lattice $L$ is isomorphic to an Alexandroff topology
if and only if for all $x \in L$,
$x = \bigvee \{ p \in \mathcal{J}_p(L) \mid p \leq x\}.$

\begin{example} \label{exa:completely_join_prime}
Let $R$ be a tolerance on $U = \{1,2,3,4\}$
such that $R(1) = \{1,2\}$, $R(2) = \{1,2,3\}$, 
$R(3) = \{2,3,4\}$, and $R(4) = \{3,4\}$. Now
\[ \wp(U)^\UP = \wp(U)^\Up = \{\emptyset, \{1,2\},  \{1,2,3\}, \{2,3,4\},
\{3,4\}, U\}.\]
This forms a nondistributive lattice, because it contains $\mathbf{N_5}$ as
a sublattice. Its completely join-irreducible elements are
$\{1,2\}$, $\{1,2,3\}$,  $\{2,3,4\}$, and $\{3,4\}$.
The completely join-irreducible elements $\{1,2,3\}$ and $\{2,3,4\}$
are not completely join-prime. For instance, 
$\{1,2,3\} \subseteq \{1,2\} \cup \{3,4\}$, but 
$\{1,2,3\} \nsubseteq \{1,2\}$ and $\{1,2,3\} \nsubseteq \{3,4\}$.
\end{example}

Because each completely join-prime element is completely join-irreducible, each completely join-prime element of $\wp(U)^\Up$ 
must be of the form $R(x)$ for some $x \in U$
(cf.~Remark~\ref{rem:join-irreducible_form}).
Therefore, the completely join-prime elements of $\wp(U)^\Up$
are elements $R(x)$ such that for all
$\mathcal{H} \subseteq \wp(U)$, 
$R(x) \subseteq \bigcup \{ Y^\Up \mid Y \in \mathcal{H}\}$
implies that $R(x) \subseteq Y^\Up$ for some 
$Y \in \mathcal{H}$.

In the following lemma, we will characterize the
completely join-prime elements similarly to Lemma~\ref{lem:irreducible}.

\begin{lemma} \label{lem:join_prime_condition}
Let $R$ be a binary relation on $U$. The following are equivalent for
$X \subseteq U$:
\begin{enumerate}[label={\rm (\roman*)}]
    \item $X^\Up$ is completely join-prime in 
    $\wp(U)^\Up$;
    \item $X^\Up = R(x)$ for some $x \in U$ such that
    for any $A \subseteq U$,
    $R(x) \subseteq \bigcup \{ R(a) \mid a \in A\}$ implies 
    $R(x) \subseteq R(b)$ for some  $b \in A$. 
\end{enumerate}
\end{lemma}

\begin{proof}
(i)$\Rightarrow$(ii):
Suppose that $X^\Up$ is completely join-prime in
$\wp(U)^\Up$. Then, $X^\Up = R(x)$
for some $x \in U$ such that
for all $\mathcal{H} \subseteq \wp(U)$, 
$R(x) \subseteq \bigcup \{ Y^\Up \mid Y \in \mathcal{H} \}$ implies 
that $R(x) \subseteq Y^\Up$ for some $Y \in \mathcal{H}$.
Let $R(x) \subseteq \bigcup \{R(a) \mid a \in A\}$ for some
$A \subseteq U$. Because each $R(a) = \{a\}^\Up$ belongs to
$\wp(U)^\Up$, we have $R(x) \subseteq R(b)$ for some $b \in A$.

\smallskip\noindent%
(ii)$\Rightarrow$(i):
Let $X^\Up = R(x)$ be such that (ii) holds.
Assume that $R(x) \subseteq \bigcup \{ Y^\Up \mid Y \in \mathcal{H} \}$
for some $\mathcal{H} \subseteq \wp(U)$. 
As shown in the proof of Lemma~\ref{lem:irreducible},
we have 
$\bigcup \{Y^\Up \mid Y \in \mathcal{H}\} = \big (\bigcup \mathcal{H} )^\Up$. Thus, our
assumption is equivalent to 
$R(x) \subseteq \big ( \bigcup \mathcal{H} \big) ^\Up$.

As we noted in the proof of Lemma~\ref{lem:irreducible},
each $X \subseteq U$ satisfies $X^\Up =  \bigcup \{ \{a\}^\Up \mid a \in X\}$. 
By this,
$R(x) \subseteq \bigcup \{ R(a) \mid a \in \bigcup \mathcal{H}\}$.
Because (ii) holds, we have that $R(x) \subseteq R(a)$ for some $a \in \bigcup \mathcal{H}$. 
This means that there is $Y \in \mathcal{H}$ such that $a \in Y$. Hence,
$R(x) \subseteq R(a) = \{a\}^\Up \subseteq Y^\Up$
and so $X^\Up = R(x)$ is completely join-prime in $\wp(U)^\Up$. 
\end{proof}

Note that completely join-prime elements have a property similar to prime numbers in arithmetic. 
If a prime number divides a product, it must divide one of the factors. 
Similarly in case of rough approximations: if a completely join-prime element $R(x)$ is 
included to the union of some $R(a)$’s, it must be included in one of them.

The following corollary is obvious by Lemma~\ref{lem:join_prime_condition}.

\begin{corollary}\label{cor:join_prime_condition2}
Let $R$ be a binary relation on $U$. 
The following are equivalent for $X \subseteq U$:
\begin{enumerate}[label={\rm (\roman*)}]
    \item $X^\UP$ is completely join-prime in $\wp(U)^\UP$;
    \item $X^\UP = \breve{R}(x)$ for some $x \in U$ such that
    for any $A \subseteq U$,
    $\breve{R}(x) \subseteq \bigcup \{ \breve{R}(a) \mid a \in A\}$ implies 
    $\breve{R}(x) \subseteq \breve{R}(b)$ for some  $b \in A$. 
\end{enumerate}
\end{corollary}

Our following lemma uses some existing results from the
literature to characterize when $\wp(U)^\Up$ is 
spatial and completely distributive.

\begin{lemma} \label{lem:compl_distri_join_prime}
Let $R$ be a binary relation on $U$. Then $\wp(U)^\Up$ is spatial and completely distributive if and only if any $R(x)$ is a union 
of such sets $R(y)$ that are completely join-prime in $\wp(U)^\Up$.
\end{lemma}

\begin{proof}
Assume that any $R(x)$ is a union of such sets $R(y)$ that are
completely join-prime. Let $X \subseteq U$.
Because $X^\Up = \bigcup \{ R(x) \mid x \in X\}$, we have that $X^\Up$
is a union of such sets $R(x)$ sets that are completely join-prime elements. 
By  \cite[Theorem 2]{Raney52} this is equivalent to the fact that 
$\wp(U)^\Up$ is isomorphic to an Alexandroff topology. 
So, $\wp(U)^\Up$ is completely distributive and spatial.

\smallskip

On the other hand, let $\wp(U)^\Up$ be completely distributive and spatial.
Because $\wp(U)^\Up$ is spatial, each element of 
$R(x) = \{x\}^\Up$ can be written as a union of completely join-irreducible elements. As we have noted, in $\wp(U)^\Up$, 
the completely join-irreducibles are  of the form $R(y)$ for some 
$y\in U$. This implies that
\[
 R(x)  =  \bigcup \{ R(y) \mid \text{ $R(y) \subseteq R(x)$ and $R(y)$ is completely join-irreducible in $\wp(U)^\Up$} \}. 
\]
Because  $\wp(U)^\Up$ is completely distributive, completely 
join-prime and completely join-irreducible elements coincide.
Therefore,
\[
 R(x)  =  \bigcup \{ R(y) \mid \text{ $R(y) \subseteq R(x)$ and $R(y)$ is completely join-prime in $\wp(U)^\Up$} \}. \qedhere
\] 
\end{proof}

We end this section by introducing the notion of the
core of $R(x)$. We begin by presenting the following
lemma.

\begin{lemma} \label{lem:CharJoinPrime}
Let $R$ be a binary relation on $U$ and $x\in U$. The following are equivalent:
\begin{enumerate}[label={\rm (\roman*)}]
\item $R(x)$ is completely join-prime in $\wp(U)^\Up$;

\item $R(x) \nsubseteq \bigcup \{R(y) \mid R(x)\nsubseteq R(y)\}$;

\item There exists $w\in R(x)$ such that $w \in R(y)$
implies $R(x)\subseteq R(y)$.
\end{enumerate}
\end{lemma}

\begin{proof}
(i)$\Rightarrow$(ii): 
Suppose $R(x) \subseteq \bigcup \{R(y) \mid R(x)\nsubseteq R(y)\}$.
Because $R(x)$ is completely join-prime in $\wp(U)^\Up$, we 
have that $R(x)\subseteq R(y)$ for some $y\in U$ with 
$R(x) \nsubseteq R(y)$, a contradiction. Thus, (ii) holds.

\smallskip%
(ii)$\Rightarrow$(iii): Suppose that  (ii) holds. There is $w\in R(x)$ such that 
$w \notin \bigcup \{R(y) \mid R(x)\nsubseteq R(y)\}$.
Therefore, for all $y \in U$, $R(x)\nsubseteq R(y)$ implies $w \notin R(y)$. 
Thus, $w \in R(y)$ implies $R(x)\subseteq R(y)$, that is, (iii) holds.

\smallskip%

(iii)$\Rightarrow $(i). 
Suppose that $R(x) = \{x\}^\Up \subseteq X^\Up$ for some $X \subseteq U$.
Assume that (iii) holds. Then, there exists $w \in R(x)$ such that
$R(x) \subseteq R(y)$ for all $R(y)$ containing $w$. Because
$X^\Up = \bigcup \{ \{x\}^\Up \mid x \in X\}$ and
$w \in R(x) \subseteq X^\Up$, there is $y \in X$ such that 
$w \in R(y)$. Because $R(y)$ contains $w$, $R(x) \subseteq R(y)$.
Hence, $R(x)$ is completely join-prime in $\wp(U)^\Up$ by
Lemma~\ref{lem:join_prime_condition}.
\end{proof}

Let $R$ be a binary relation on $U$ and $x \in U$. 
The \emph{core of $R(x)$}
is defined by 
\begin{equation} \label{eq:core}
\mathfrak{core}R(x) :=  \{ w \in R(x) \mid 
w \in R(y) \text{ implies } R(x) \subseteq R(y)\}.
\end{equation}
We define the \emph{core of $\breve{R}(x)$} by 
writing  $\breve{R}$ instead of $R$ in \eqref{eq:core}.
Our next lemma is clear from Lemma~\ref{lem:CharJoinPrime}. It shows the direct connection between completely join-prime elements and the neighbourhoods 
whose core is nonempty.

\begin{lemma} \label{lem:core_irreducibles}
Let $R$ be a binary relation on $U$.
\begin{enumerate}[label={\rm (\roman*)}]
\item $\mathcal{J}_p(\wp(U)^\Up) = 
\{ R(x) \mid \mathfrak{core}R(x) \neq \emptyset\}$;
\item $\mathcal{J}_p(\wp(U)^\UP) = 
\{ \breve{R}(x) \mid \mathfrak{core}\breve{R}(x) \neq \emptyset\}$.
\end{enumerate}
\end{lemma}

\begin{example} \label{exa:compl_distr}
The reflexive relation $R$ on $U = \{1, 2, 3, 4\}$, visualized in Figure~\ref{fig:directed_similarity}, is a directional similarity relation.
\begin{figure}[h]
\includegraphics[width=20mm]{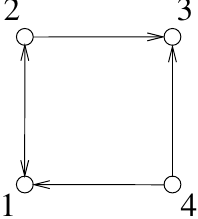}
\caption{The directional similarity relation $R$.}
\label{fig:directed_similarity}
\end{figure}
If an element $a$ is $R$-similar to $b$, then there is an
arrow from $a$ to $b$. Because $R$ is reflexive, there should be an arrow from each circle to the circle itself. Such loops can be omitted.

Now, element $1$ is similar to $2$, and $2$ is similar to $1$ and
$3$. Element $3$ is not similar to other elements, so it
is a singleton. Element $4$ is similar to $1$ and $4$. The
$R$-neighbourhoods are the following sets:
\begin{gather*}
R(1) = \{1,2\}, \ R(2) = \{1,2,3\}, \ R(3) = \{3\}, \ 
R(4) = \{1,3,4\}.
\end{gather*}
The relation $R$ is reflexive. It is not symmetric, 
because $(2,3) \in R$, but $(3,2) \notin R$. The relation is
not transitive, because $(1,2)$ and $(2,3)$ belong to $R$,
but $(1,3)$ is not in $R$.

The lattices $\wp(U)^\Up$ and $\wp(U)^\UP$ are depicted in Figure~\ref{fig2:upperwhite}. They are (completely) distributive.
The elements $R(1)$, $R(3)$, and $R(4)$ are completely 
join-irreducible
and completely join-prime in $\wp(U)^\Up$. 
Similarly, $\breve{R}(2)$, $\breve{R}(3)$, and 
$\breve{R}(4)$ are  completely  join-irreducible
and completely join-prime in $\wp(U)^\UP$. 

\begin{figure}[h]
\includegraphics[width=120mm]{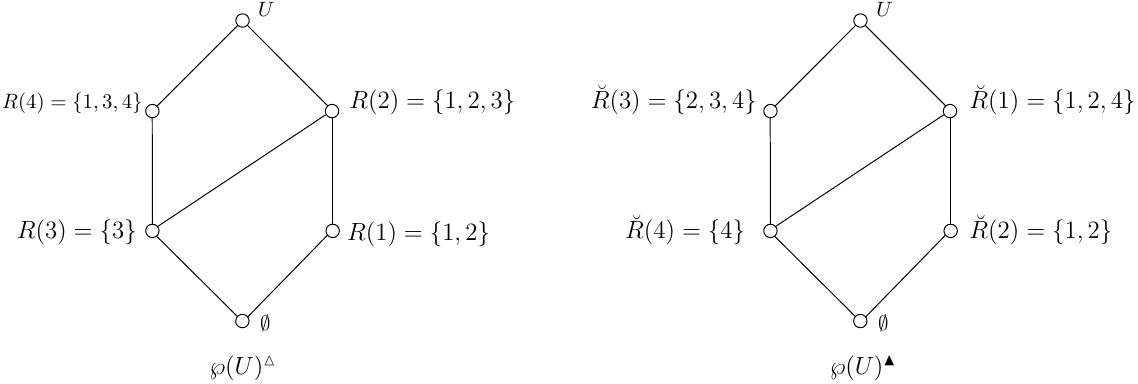}
\caption{The distributive lattices
$\wp(U)^\Up$ and $\wp(U)^\UP$ of Example~\ref{exa:compl_distr}.}
\label{fig2:upperwhite}
\end{figure}

\noindent%
We have now that
\[ 
\mathfrak{core}R(1) = \{2\}, \
\mathfrak{core}R(2) = \emptyset, \
\mathfrak{core}R(3) = \{3\}, \
\mathfrak{core}R(4) = \{4\};
\]
and
\[
\mathfrak{core}\breve{R}(1) = \emptyset, \
\mathfrak{core}\breve{R}(2) = \{1\}, \
\mathfrak{core}\breve{R}(3) = \{3\}, \
\mathfrak{core}\breve{R}(4) = \{4\}.
\]
Note that $1 \notin \mathfrak{core}R(1)$. This is because 
$1 \in R(4)$, but $R(1) \nsubseteq R(4)$. This means that
there are such elements $x$ that $x \notin \mathfrak{core}R(x)$ even
$\mathfrak{core}R(x) \neq \emptyset$ and 
$\wp(U)^\Up$ is completely distributive and spatial.
Similarly, $2 \notin \mathfrak{core}\breve{R}(2)$
concerning the inverse of $R$.
\end{example}

The following result characterizes the case when $\mathrm{DM(RS)}$ is completely distributive and spatial in terms of cores.
The result is obvious by Propositions
\ref{prop:completely_distributive} and
\ref{prop:spatial}, and 
Lemmas \ref{lem:compl_distri_join_prime}
and \ref{lem:core_irreducibles}.

\begin{corollary}\label{cor:spatial_reflexive}
Let $R$ be a binary relation on $U$. 
$\mathrm{DM(RS)}$ is spatial and completely distributive
if and only if the following equations hold for each $x \in U$:
\begin{enumerate}[label={\rm (\roman*)}]
\item 
$R(x) = \bigcup \{ R(y) \mid
R(y) \subseteq R(x) \text{ and } 
\mathfrak{core}R(y) \neq \emptyset \}$;

\item 
$\breve{R}(x) = \bigcup \{ \breve{R}(y) \mid
\breve{R}(y) \subseteq \breve{R}(x) \text{ and } 
\mathfrak{core}\breve{R}(y) \neq \emptyset \}$.
\end{enumerate}
\end{corollary} 
 
Note also that spatial and completely distributive lattices are
isomorphic to some Alexandroff topologies; see 
\cite[Theorem~10.29]{Davey02}.

In the following lemma, we present some useful properties of
the cores. Similar facts can be proved also for $\mathfrak{core}\breve{R}(x)$
by interchanging $R(x)$ and $\breve{R}(x)$. Note that the inverse of
$\breve{R}$ is $R$ again. 

\begin{lemma}\label{lem:core}
Let $R$ be a binary relation on $U$.
\begin{enumerate}[label={\rm (\roman*)}]
\item $R(x)=R(y)$ implies $\mathfrak{core}R(x)=
\mathfrak{core}R(y)$.

\item $w\in \mathfrak{core}R(x)$ if and only if 
$w\in R(x)$ and $\breve{R}(w)\subseteq \breve{R}(z)$ 
for all $z\in R(x)$.

\item If $y_{1},y_{2}\in \mathfrak{core}R(x)$, then 
$\breve{R}(y_{1}) = \breve{R}(y_{2})$.

\item If $y \in \mathfrak{core}R(x)$, then $x\in \mathfrak{core}\breve{R}(y)$.

\item If $R$ is reflexive, then $y \in \mathfrak{core}R(x)$ 
implies $R(x) \subseteq R(y)$. 
\end{enumerate}
\end{lemma}

\begin{proof}
(i) If $R(x)=R(y)$, then by the definition of the core, 
$w\in \mathfrak{core}R(x)$ implies $w\in \mathfrak{core}R(y)$.
Conversely, $w\in \mathfrak{core}R(y)$ implies $w\in \mathfrak{core}R(x)$. Hence, $\mathfrak{core}R(x)=\mathfrak{core}R(y)$.

(ii) If $w\in \mathfrak{core}R(x)$, then $w\in R(x)$ and $R(x)\subseteq R(y)$ for all $y$ with $w\in R(y)$. 
Let $z\in R(x)$. If $y\in \breve{R}(w)$, then 
$w\in R(y)$ and $z\in R(x)\subseteq R(y)$, that is, 
$y\in \breve{R}(z)$. This means that 
$\breve{R}(w)\subseteq \breve{R}(z)$ for any $z\in R(x)$.

Conversely, assume that $w\in R(x)$ and 
$\breve{R}(w)\subseteq \breve{R}(z)$ for all $z\in R(x)$. 
Let $w\in R(y)$ for some $y\in U$. Then, 
$y\in \breve{R}(w)\subseteq \breve{R}(z)$ for all $z\in R(x)$. 
Hence, $z\in R(y)$ for all $z\in R(x)$, that is, 
$R(x)\subseteq R(y)$. This means that 
$w\in \mathfrak{core}R(x)$.

(iii) Since $y_{1},y_{2}\in \mathfrak{core} R(x) \subseteq R(x)$, 
we have $\breve{R}(y_{1})\subseteq \breve{R}(y_{2})$ and 
$\breve{R}(y_{2}) \subseteq \breve{R}(y_{1})$
by (ii).
Hence, $\breve{R}(y_{1})=\breve{R}(y_{2})$.

(iv) Assume that $y \in \mathfrak{core} R(x)$.
To prove that $x \in \mathfrak{core} \breve{R}(y)$, we need to
show that $\breve{R}(y) \subseteq \breve{R}(z)$ for all
$\breve{R}(z)$ such that $x \in \breve{R}(z)$.
For this, let $z$ be such that $x \in \breve{R}(z)$. 
Now $y \in \mathfrak{core} R(x)$ implies
$\breve{R}(y) \subseteq \breve{R}(z)$ by (ii) since
$z \in R(x)$. 

(v) Let $R$ be reflexive.
If $y \in \mathfrak{core} R(x)$, then $y \in R(y)$ implies 
$R(x) \subseteq R(y)$ by the definition of the core.
\end{proof}

We will conclude  this section by showing that 
the ordered sets of the completely join-prime elements of $\wp(U)^\Up$ and $\wp(U)^\UP$ are dually order-isomorphic. To achieve this, let us define a pair of mappings between  
$\mathcal{J}_{p}(\wp(U)^\Up)$ and
$\mathcal{J}_{p}(\wp(U)^\UP)$ by setting:
\[ \varphi \colon \mathcal{J}_{p}(\wp(U)^\Up)
\to
\mathcal{J}_{p}(\wp(U)^\UP), \
R(x) \mapsto \breve{R}(w_x), 
\text{where $w_x \in \mathfrak{core}R(x)$}
\]
and 
\[ \psi \colon \mathcal{J}_{p}(\wp(U)^\UP)
\to
\mathcal{J}_{p}(\wp(U)^\Up), \
\breve{R}(y) \mapsto R(w_y), 
\text{where $w_y \in \mathfrak{core}\breve{R}(y)$}.
\]

\begin{lemma} \label{lem:order_reversing}
Let $R$ be a binary relation on $U$.
\begin{enumerate}[label={\rm (\roman*)}]
\item The maps $\varphi$ and $\phi$ are well-defined.
\item The maps $\varphi$ and $\psi$ are $\subseteq$-reversing.
\item The maps $\varphi$ and $\psi$ are inverses of each other.
\end{enumerate}
\end{lemma}

\begin{proof} (i)
If $R(x)\in\mathcal{J}_{p}(\wp(U)^\Up)$, 
$\mathfrak{core}R(x)$
has at least one element $w_x\in\mathfrak{core}R(x)$. 
By Lemma~\ref{lem:core}(iv), 
$w_x\in\mathfrak{core}R(x)$ implies 
$x\in\mathfrak{core}\breve{R}(w_x)$. Thus, 
$\mathfrak{core}\check{R}(w_x)\neq\emptyset$. Hence,
$\varphi(R(x))$ exists and belongs to
$\mathcal{J}_{p}(\wp(U)^\UP)$.
For any $w_{1},w_{2}\in\mathfrak{core}R(x)$, we have
$\check{R}(w_{1})=\check{R}(w_{2})$
by Lemma~\ref{lem:core}(iii). 
Thus, the element $\varphi(R(x))$ is unique, and
$\varphi$ is well defined.

That $\psi$ is well-defined map can be proved analogously. 
 
(ii)
Take any $R(x_1),R(x_2)\in\mathcal{J}_{p}(\wp(U)^{\Up})$ such
that $R(x_1)\subseteq R(x_2)$. We  prove that
$\varphi(R(x_1)) \supseteq \varphi (R(x_2))$. 
By definition, $\varphi(R(x_1)) = \breve{R}(w_1)$, where $w_1\in\mathfrak{core}R(x_1)$ and 
$\varphi(R(x_2)) = \breve{R}(w_2)$, where
$w_2\in\mathfrak{core}R(x_2)$.

Let $y\in\breve{R}(w_2)$. Then $w_2 \in R(y)$
and $w_2\in\mathfrak{core}R(x_2)$ imply 
$R(x_2)\subseteq R(y)$. 
This means that 
$w_1 \in \mathfrak{core}R(x_1) \subseteq
R(x_1) \subseteq R(x_2) \subseteq R(y)$. 
This yields $y\in\breve{R}(w_1)$. We have proved
$\breve{R}(w_2) \subseteq \breve{R}(x_2)$, that is, 
$\varphi(R(x_1))\supseteq \varphi(R(x_1))$.

That $\phi$ is order-reversing can be proved in
a similar manner.

(iii) ) Assume $R(x)\in\mathcal{J}_{p}(\wp(U)^\Up$. Then,
$\mathfrak{core}R(x)\neq\emptyset$ and 
$\varphi(R(x)) = \breve{R}(w_x)$ for some
$w_x\in\mathfrak{core}R(x)$. 
We have that $x\in\mathfrak{core}\breve{R}(w_x)$ and so
$\psi(\breve{R}(w_x))=R(x)$. Therefore,
$\psi(\varphi(R(x)))=R(x)$.

Analogously, we can show that 
$\varphi(\psi(\breve{R}(y)))=\breve{R}(y)$ for 
each $\breve{R}(y) \in \mathcal{J}_p(\wp(U)^\UP)$.
Therefore, we have proved that 
$\varphi$ and $\psi$ are inverses of each other
\end{proof}

Now we are ready to present the isomorphism.

\begin{proposition}
If $R$ is a binary relation on $U$, then
\[ 
(\mathcal{J}_{p}(\wp(U)^\Up),\subseteq) \cong
(\mathcal{J}_{p}(\wp(U)^\UP), \supseteq).
\]
\end{proposition}

\begin{proof}
The result follows from Lemma~\ref{lem:order_reversing}.
Since $\varphi$ and $\psi$ are inverses of each other,
they are bijections. 
Because $\varphi$ is a bijective order-preserving map from 
$(\mathcal{J}_{p}(\wp(U)^\Up),\subseteq)$ to
$(\mathcal{J}_{p}(\wp(U)^\UP), \supseteq)$
with  an order-preserving inverse $\psi$, it
is an order isomorphism. The same holds for $\psi$.
\end{proof}

\section{Nelson algebras defined on $\mathrm{DM(RS)}$}
\label{Sec:NelsonAlgebras}

As we mentioned in the introduction, $\mathrm{RS}$ forms a completely distributive and spatial Nelson algebra when induced
by a quasiorder.
Conversely, for each completely distributive and spatial Nelson algebra, there exists a quasiorder on $\mathcal{J}$ such that the induced rough set algebra is isomorphic to the original Nelson algebra. Therefore, Nelson algebras appear to be closely connected to quasiorders. In this section, we reveal a new condition under which certain relations that are merely reflexive induce Nelson algebras.

A \emph{De~Morgan algebra} $(L,\vee,\wedge,{\sim},0,1)$ is
an algebra such that $(L,\vee,\wedge,0,1)$  
is a bounded distributive lattice and the negation $\sim$ satisfies the \emph{double negation law}
\[ {\sim} {\sim} x = x,\] 
and the two \emph{De Morgan laws} 
\[ {\sim} (x \vee y) = {\sim} x \wedge {\sim} y  
\text{ \ and \ } 
{\sim} (x \wedge y) = {\sim} x \vee {\sim} y .  \]
Note that this means that $\sim$ is an order-isomorphism
between $(L,\leq)$ and $(L,\geq)$.

We say that a De Morgan algebra is \emph{completely distributive} if its underlying lattice is completely distributive. 
Let $(L,\vee,\wedge,{\sim},0,1)$ be a completely 
distributive De Morgan algebra. We define for any $j \in \mathcal{J}$ the element
\begin{equation}\label{Eq:Gee}
 g(j)= \bigwedge \{x \in L \mid x\nleq {\sim} j \}.
\end{equation}
This $g(j) \in \mathcal{J}$ is the least element which is not below ${\sim}j$. The function $g \colon \mathcal{J} \to \mathcal{J}$
satisfies the conditions:
\begin{enumerate}[label = ({J}\arabic*)]
 \item if $x \leq y$, then $g(x) \geq g(y)$;
 \item $g(g(x))= x$.
\end{enumerate}
In fact, $(\mathcal{J},\leq)$ is self-dual by the map $g$.
For studies  on the properties of the map $g$, see  
\cite{Cign86, JR11, Mont63a},
for example.

\medskip%

A \emph{Kleene algebra} is a De Morgan algebra 
satisfying inequality
\begin{equation} \label{Eq:Kleene}\tag{K}
x \wedge {\sim} x \leq y \vee {\sim} y.
\end{equation}
Let $(L,\vee,\wedge,{\sim},0,1)$ be a completely distributive 
Kleene algebra. 
Then, $j$ and $g(j)$ are comparable for any $j \in \mathcal{J}$, that is,
\begin{enumerate}[label = ({J}\arabic*)]
\addtocounter{enumi}{2}
 \item $g(j)\leq j \text{ or } j \leq g(j)$.
\end{enumerate}
We may define three disjoint sets:
\begin{align*}
    \mathcal{J}^- &= \{j \in \mathcal{J} \mid j < g(j) \}; \\
    \mathcal{J}^\circ &= \{j \in \mathcal{J} \mid j = g(j) \}; \\
    \mathcal{J}^+ &= \{j \in \mathcal{J} \mid j > g(j) \}.
\end{align*}
The following simple lemma gives some useful
conditions for $\mathcal{J}^-$ and $\mathcal{J}^+$.

\begin{lemma}\label{lem:KleeneNeg}
Let $(L,\vee,\wedge,{\sim},0,1)$  be a completely distributive
Kleene algebra.
\begin{enumerate}[label={\rm (\roman*)}]
\item $j\in \mathcal{J}^{-} \iff  g(j)\in \mathcal{J}^{+}$;  
\item $\mathcal{J}^{-} = \{ j \in \mathcal{J} \mid 
j \leq {\sim} j\}$.
\end{enumerate}
\end{lemma}

\begin{proof} 
(i) Suppose $j \in \mathcal{J}^-$. Then, $j < g(j)$ and
$g(j) > g(g(j)) = j$, that is, $g(j)\in \mathcal{J^+}$.
The other direction may be proved similarly.

(ii) If $j \in \mathcal{J}$ is such that 
$j \nleq {\sim} j$, then $g(j) \leq j$, because $g(j)$
is the least element which is not below ${\sim}j$. 
If $j \in \mathcal{J}^-$, then $j < g(j)$.
Therefore, $g(j) \nleq j$ and we must have $j \leq {\sim}j$.
Thus, $\mathcal{J}^{-} \subseteq \{ j \in \mathcal{J} \mid 
j \leq {\sim} j\}$;

On the other, if $j \leq {\sim}j$, then $g(j) \leq j$ is not possible. Therefore, $g(j) > j$ and $j \in \mathcal{J}^-$.
\end{proof}

Our following proposition is clear by \cite[Lemma~2.1]{JarRad23}.

\begin{proposition} \label{Prop:Kleene}
If $R$ is a reflexive relation such that $\mathrm{DM(RS)}$ is
a distributive lattice, then
\[ (\mathrm{DM(RS)},\vee,\wedge,{\sim},(\emptyset,\emptyset), (U,U) ) \]
is a Kleene algebra in which ${\sim}(A,B) = (B^c, A^c)$ for all
$(A,B) \in  \mathrm{DM(RS)}$.
\end{proposition}
We will next describe the sets $\mathcal{J}^{-}$,
$\mathcal{J}^\circ$, and $\mathcal{J}^+$ in the case
$\mathrm{DM(RS)}$ forms a completely distributive
lattice.

\begin{proposition} \label{prop:partition_of_J}
Let $R$ be a reflexive relation on $U$ such that 
$\mathrm{DM(RS)}$ is completely distributive. Then, the
following assertions hold:
\begin{enumerate}[label={\rm (\roman*)}, itemsep=4pt]
\item %i
$\mathcal{J}^{-} = 
\{(\emptyset ,\{x\}^\UP) \mid \{x\}^\UP \in
\mathcal{J}(\wp (U)^{\UP}) \text{ and } x \notin \mathcal{S} \}$.

\item %ii
If $(\emptyset ,\{x\}^\UP) \in \mathcal{J}^{-}$, then
$g(\emptyset,\{x\}^\UP) =  (\{z\}^{\Up \DOWN}, \{z\}^{\Up \UP})$
for any $z\in \mathfrak{core}\breve{R}(x)$, $z \notin \mathcal{S}$,
and $\{z\}^\Up$ is completely join-irreducible in  
$\wp(U)^\Up$.

\item %iii
$\mathcal{J}^{+} = \{(\{x\}^{\Up \DOWN},\{x\}^{\Up \UP}) \mid 
\{x\}^\Up \in \mathcal{J}(\wp(U)^{\Up}) \text{ and }
x \notin \mathcal{S} \}$;

\item
$\mathcal{J}^{\circ}$ = $\{(\{x\},\{x\}^\UP)\mid x\in \mathcal{S} \}$.
\end{enumerate}

\bigskip     
\end{proposition}

\begin{proof}
(i) Let $j=(A,B)\in \mathcal{J}^{-}$. Then, 
by Lemma~\ref{lem:KleeneNeg}(ii),
$j\leq {\sim} j$ and 
\[
j = j \wedge {\sim} j = (A,B)\wedge (B^{c},A^{c}) 
  = (A \cap B^c, (B \cap A^c)^{\Down \UP}) 
  = (\emptyset, (B \setminus A)^{\Up\DOWN}),
\]
because $A\subseteq B$. In view of Theorem~\ref{thm:join_irreducibles}, this implies that $j=(\emptyset,\{x\}^{\UP})$ for some 
$x \notin \mathcal{S}$ such that $\{x\}^\UP$ is completely join-irreducible in $\wp (U)^\UP$.

Conversely, if $j = (\emptyset ,\{x\}^{\blacktriangle })$, where 
$\{x\}^\UP$ is completely join-irreducible in $\wp(U)^\UP$ and 
$x \notin \mathcal{S}$, then $j \in \mathcal{J}$
by Proposition~\ref{prop:join_irred}.
We have  
${\sim} j = (U \setminus \{x\}^\UP,U)$ and $j \leq {\sim} j$.
We obtain $j\in \mathcal{J}^{-}$ by
Lemma~\ref{lem:KleeneNeg}(ii).

(ii) As $(\emptyset ,\{x\}^\UP) \in \mathcal{J}^-$,
Theorem~\ref{thm:join_irreducibles} gives
$x\notin \mathcal{S}$.
Lemma~\ref{lem:KleeneNeg}(i) yields that 
$g(\emptyset,\{x\}^\UP)\in \mathcal{J}^{+}$. Therefore, 
in view of Theorem~\ref{thm:join_irreducibles}, we get 
$g(\emptyset ,\{x\}^\UP) = (\{z\}^{\Up \DOWN}, \{z\}^{\Up\UP})$ 
for some $z\in U$.
We will show that $z \in \mathfrak{core}\breve{R}(x)$, $z \notin \mathcal{S}$,
and $\{z\}^\Up$ is completely join-irreducible in 
$\wp(U)^\Up$.

Let us observe that for any $a \in U$, 
$\{a\}^{\Up\DOWN} \cap \{x\}^\UP \neq \emptyset$ 
is equivalent to $a \in \{x\}^\UP = \breve{R}(x)$. 
This can be seen as follows. Since $a\in \{a\}^{\Up \DOWN}$,
$a\in \{x\}^\UP$ implies 
$\{a\}^{\Up \DOWN} \cap \{x\}^\UP \neq \emptyset$. 
Conversely, $\{a\}^{\Up \DOWN} \cap \{x\}^\UP \neq \emptyset$
implies that there exists $b\in \{a\}^{\Up \DOWN} \cap \breve{R}(x)$.
Thus, $x \in R(b)$. Because $b \in \{a\}^{\Up \DOWN}$,
we have $R(b) \subseteq \{a\}^\Up = R(a)$. We have
$x \in R(a)$ and  $a \in \breve{R}(x) = \{x\}^\UP$. 

Because $g(j) \nleq {\sim} j$, we get 
$(\{z\}^{\Up \DOWN}, \{z\}^{\Up\UP}) \nleq (U\setminus \{x\}^\UP, U)\}$,
which is equivalent to 
$\{z\}^{\Up\DOWN} \cap \{x\}^\UP \neq \emptyset$. 
Hence $z\in \breve{R}(x) = \{x\}^\UP$. 
Let $a\in \breve{R}(x)$. Then $\{a\}^{\Up \DOWN} \cap \{x\}^\UP \neq \emptyset$ holds. We have $\{a\}^{\Up \DOWN} \nsubseteq \{x\}^{\UP c}$ 
and
\[ (\{a\}^{\Up \DOWN}, \{a\}^{\Up\UP}) \nleq (\{x\}^{\UP c}, U) =
{\sim} (\emptyset, \{x\}^\UP). \]
Because $g(\emptyset,\{x\}^\UP)$ is the least completely 
join-irreducible element which is not below 
${\sim} (\emptyset, \{x\}^\UP)$, we have that 
$(\{z\}^{\Up \DOWN}, \{z\}^{\Up\UP}) \leq 
(\{a\}^{\Up \DOWN}, \{a\}^{\Up\UP})$. Then 
$\{z\}^{\Up \DOWN} \subseteq \{a\}^{\Up \DOWN}$ yields 
\[ R(z) = \{z\}^\Up= \{z\}^{\Up \DOWN \Up}
\subseteq \{a\}^{\Up \DOWN \Up} = \{a\}^\Up = R(a).
\]
Summarising, we obtained that $z\in \breve{R}(x)$ and $R(z)\subseteq R(a)$ 
for all $a\in \breve{R}(x)$. If we replace $R$ with $\breve{R}$ in
Lemma~\ref{lem:core}(ii), we get 
$z\in \mathfrak{core}\breve{R}(x)$. 
Since $x\in R(z)$, $z$ cannot be in $\mathcal{S}$. Indeed, 
$z \in \mathcal{S}$ would imply $x=z$ contradicting
$x \notin \mathcal{S}$. 

Because $(\emptyset,\{x\}^\UP)$ is
completely join-irreducible, $\{x\}^\UP$ is completely
join-irreducible and completely join-prime in $\wp(U)^\UP$.
Since $z \in \mathfrak{core}\breve{R}(x)$, 
then, by Lemma~\ref{lem:core}(iv), $x \in \mathfrak{core}R(z)$.
Since the core of $R(z)$ is nonempty, we have that
$R(z) = \{z\}^\Up$ is completely join-irreducible in $\wp(U)^\Up$.

(iii) Let $j\in \mathcal{J}^{+}$. Then, by Lemma~\ref{lem:KleeneNeg}, 
$g(j) \in \mathcal{J}^{-}$. By (i), 
$g(j) = (\emptyset ,\{x\}^\UP)$ for some $x\notin \mathcal{S}$
such that $\{x\}^\UP$ is completely join-irreducible in $\wp(U)^\UP$. 
In view of (ii), we obtain 
$j = g(g(j)) = (\{z\}^{\Up \DOWN}, \{z\}^{\Up\UP})$ for some
$z \notin \mathcal{S}$ such that
$\{z\}^\Up$ is completely join-irreducible in $\wp (U)^\Up$.

(iv) Since $j = g(j)$, we have $j \notin \mathcal{J}^{-}$. 
This means that $j=(\{x\}^{\Up \DOWN},\{x\}^{\Up\UP})$, where 
$\{x\}^\Up$ is completely join-irreducible in $\wp(U)^\Up$. 
Because $j\notin \mathcal{J}^{+}$, we must have $x \in \mathcal{S}$ and 
$R(x) = \{x\}^\Up = \{x\}$. Now
$\{x\}^{\Up\DOWN} = \{x\}^\DOWN = \{x\}$ and
$\{x\}^{\Up\UP} = \{x\}^\UP$.
\end{proof}

A \emph{Heyting algebra} $L$ is a bounded lattice such that for all 
$a,b\in L$, there is a greatest element $x \in L$ such that
\[
a\wedge x\leq b.
\]
This element $x$ is the \emph{relative pseudocomplement} of $a$ with
respect to $b$, and it is denoted by $a\Rightarrow b$. It is well known that
any completely distributive lattice $L$ determines a Heyting algebra
$(L,\vee,\wedge,\Rightarrow,0,1)$. 

According to R.~Cignoli \cite{Cign86}, a \emph{quasi-Nelson algebra} \label{def:quasi-Nelson}
is a Kleene algebra  $(L,\vee,\wedge,\sim,0,1)$
such that for  all $a,b \in A$, the \emph{weak relative pseudocomplement} of $a$ with respect to $b$
\begin{equation} \label{eq:NelsonImplication}
a \to b :=  a \Rightarrow ({\sim}a \vee b)
\end{equation}
exists. This means that every Kleene algebra whose underlying lattice is a Heyting algebra forms a quasi-Nelson algebra. In particular, 
any Kleene algebra defined on a completely distributive lattice is a 
quasi-Nelson algebra.
Thus, Propositions \ref{prop:completely_distributive} 
and \ref{Prop:Kleene} have the following corollary.

\begin{corollary}\label{cor:quasi_nelson}
Let $R$ be a reflexive relation on $U$. 
If any of $\wp(U)^\UP$, $\wp(U)^\DOWN$,  $\wp(U)^\Up$, 
$\wp(U)^\Down$  is completely distributive, then
$\mathrm{DM(RS)}$ forms a quasi-Nelson algebra.
\end{corollary}

Let $(L,\vee,\wedge,{\sim},0,1)$ be a completely distributive Kleene algebra. We say that the set $\mathcal{J}$ of its completely 
join-irreducible elements satisfies the \emph{interpolation property} if for any $p,q \in \mathcal{J}$ such that $p,q \leq g(p),g(q)$, there is $k \in \mathcal{J}$ such that
\[p,q \leq k \leq g(p),g(q)  .\]

An element $a$ of a complete lattice $L$ is called \emph{compact} if
$a\leq \bigvee X$ for some $X\subseteq L$ implies that 
$a\leq X^{\bullet}$ for some finite subset $X^{\bullet} \subseteq X$. 
A complete lattice is \emph{algebraic} if its every element $x$ is the supremum of the compact elements below $x$.
By  \cite[Theorem~10.29]{Davey02}, 
a De~Morgan algebra is algebraic if and only it is 
completely distributive and spatial.
 
In \cite[Proposition~3.5]{JR11} we proved that if 
$(L,\vee, \wedge, {\sim}, 0, 1)$ is a Kleene algebra defined on an algebraic lattice, then $(L,\vee, \wedge, \to, {\sim}, 0, 1)$
is a Nelson algebra, where the operation $\to$ is defined by \eqref{eq:NelsonImplication} if and only if $\mathcal{J}$ satisfies the interpolation property.

The following lemma is useful when proving 
Theorem~\ref{thm:characterization}.

\begin{lemma} \label{lem:pre_interpolation}
Suppose that $R$ is a reflexive relation on $U$
such that $\mathrm{DM(RS)}$ is completely distributive. Let
$\{x\}^\UP$ and $\{y\}^\UP$ be completely join-irreducible elements of 
$\wp(U)^\UP$ and $x,y\notin \mathcal{S}$. 
The following assertions are equivalent:

\begin{enumerate}[label={\rm (\roman*)}, itemsep=4pt]
\item There exists an element $u\in U$ with 
$\{x\}^\UP,\{y\}^\UP \subseteq \{u\}^\UP$.

\item $(\emptyset,\{x\}^\UP),(\emptyset,\{y\}^\UP)
\leq g(\emptyset,\{x\}^\UP), g(\emptyset,\{y\}^\UP)$.
\end{enumerate}
\end{lemma}

\begin{proof}
In view of Proposition~\ref{prop:partition_of_J}(i), 
$(\emptyset,\{x\}^\UP)$ and $(\emptyset,\{y\}^\UP)$ belong to the set
$\mathcal{J}^-$. Furthermore, 
$g(\emptyset,\{x\}^\UP)=(\{v_{x}\}^{\Up\DOWN},\{v_x\}^{\Up\UP})$ and $g(\emptyset,\{y\}^{\UP})=(\{v_{y}\}^{\Up\DOWN},\{v_{y}\}^{\Up\UP})$, 
where $v_{x}\in \mathfrak{core}\breve{R}(x)$ and 
$v_{y}\in \mathfrak{core}\breve{R}(y)$ can be chosen arbitrarily 
in virtue of Proposition~\ref{prop:partition_of_J}(ii)

\smallskip\noindent%
(i)$\Rightarrow$(ii). Condition (i) implies $v_{x}\in \breve{R}(x) =
\{x\}^\UP \subseteq\{u\}^\UP = \breve{R}(u)$ and
$v_{y}\in \breve{R}(y)=\{y\}^\UP \subseteq \{u\}^\UP =\breve{R}(u)$, 
whence we get $\{u\}\subseteq R(v_{x})=\{v_{x}\}^\Up$
and $\{u\} \subseteq R(v_{y})=\{v_{y}\}^\Up$. Hence,
$\{x\}^\UP,\{y\}^\UP\subseteq\{u\}^\UP \subseteq
\{v_{x}\}^{\Up\UP},\{v_{y}\}^{\Up\UP}$, and this yields 
$(\emptyset,\{x\}^\UP),(\emptyset,\{y\}^\UP) \leq 
(\{v_{x}\}^{\Up\DOWN},\{v_{x}\}^{\Up\UP}),
(\{v_{y}\}^{\Up\DOWN}, \{v_{y}\}^{\Up\UP})$, which
is equivalent to (ii).

\smallskip\noindent%
(ii)$\Rightarrow$(i). If $(\emptyset,\{y\}^\UP)\leq
g(\emptyset,\{x\}^\UP)$, then $\{y\}^\UP
\subseteq \{v_{x}\}^{\Up\UP}$, that is, 
\[ \{y\}^\UP \subseteq R(v_{x})^\UP= 
\bigcup \{  \{z\}^\UP \mid z\in R(v_{x})\}.\] 
Since $\{y\}^\UP$ is completely join-prime in $\wp(U)^\UP$,
$\{y\}^\UP \subseteq \{u\}^\UP$ for some $u \in R(v_{x})$.
Moreover, we have $v_x \in \breve{R}(u)$. Because 
$v_x \in \mathfrak{core}\breve{R}(x)$, we have that $v_x \in \breve{R}(u)$
implies $\breve{R}(x) \subseteq \breve{R}(u)$ by the
definition of the core, that is, $\{x\}^\UP \subseteq \{u\}^\UP$.
We have now proved 
$\{x\}^\UP,\{y\}^\UP \subseteq \{u\}^\UP$.
\end{proof}

The next theorem is our main result of this section.

\begin{theorem} \label{thm:characterization}
Let $R$ be a reflexive relation on $U$ such that 
$\mathrm{DM(RS)}$ is completely distributive and
spatial. The following assertions are equivalent:

\begin{enumerate}[label={\rm (\roman*)}, itemsep=4pt]
\item $\mathrm{DM(RS)}$ forms a Nelson algebra.

\item Let $x,y \notin \mathcal{S}$ be such that
$\{x\}^\UP,\{y\}^\UP$ are completely join-irreducible in $\wp(U)^\UP$.
If $\{x\}^\UP,\{y\}^\UP \subseteq \{u\}^\UP$ for some 
$u\in U$, then there exists an element $z\in U$ such that 
$\{z\}^\UP$ is completely join-irreducible in $\wp(U)^\UP$ and
$\{x\}^\UP,\{y\}^\UP \subseteq \{z\}^\UP$.
\end{enumerate}
\end{theorem}

\begin{proof} 
(i)$\Rightarrow$(ii). Assume that (i) holds. Let $x,y,u$ be as in 
the assumption of (ii). Then, by Proposition~\ref{prop:partition_of_J}, 
$(\emptyset,\{x\}^\UP)$ and $(\emptyset,\{y\}^\UP)$ belong to $\mathcal{J}^{-}$. Lemma~\ref{lem:pre_interpolation} implies that
$(\emptyset,\{x\}^\UP)$, $(\emptyset,\{y\}^\UP)\leq
g(\emptyset,\{x\}^\UP)$, $g(\emptyset,\{y\}^\UP)$. Let us
set $p:=(\emptyset,\{x\}^\UP)$, $q:=(\emptyset,\{y\}^\UP)$. Then, 
$p,q\leq g(p),g(q)$. As $\mathrm{DM(RS)}$ is a Nelson algebra, 
by the interpolation property, there exists $k\in\mathcal{J}$ such
that $p,q\leq k\leq g(p),g(q)$. Then $p,q\leq g(k)\leq g(p),g(q)$ also holds.

If $k\in\mathcal{J}^{\circ}$, then in view of 
Proposition~\ref{prop:partition_of_J},
$k=(\{z\},\{z\}^\UP)$ for some $z\in \mathcal{S}$. Therefore,
$\{z\}^\UP$ is a completely join-irreducible element of $\wp(U)^\UP$ by
Lemma~\ref{lem:singleton_join_irr} and 
$\{x\}^\UP, \{y\}^\UP \subseteq \{z\}^\UP$.
If $k \in \mathcal{J}^-$, then $k = (\emptyset,\{z\}^\UP)$ for some
$\{z\}^\UP \in \mathcal{J}(\wp(U)^\UP)$. Now $p,q \leq k$ gives 
$\{x\}^\UP, \{y\}^\UP \subseteq \{z\}^\UP$. 
Finally, if $k \in \mathcal{J}^+$, then 
$g(k) \in \mathcal{J}^-$. Thus, $g(k) =  (\emptyset,\{z\}^\UP)$ for some
$\{z\} \in \mathcal{J}(\wp(U)^\UP)$. Since $p,g \leq g(k)$,
we have $\{x\}^\UP, \{y\}^\UP \subseteq \{z\}^\UP$. 
Thus (ii) holds in all possible cases.

\medskip\noindent%
(ii)$\Rightarrow$(i). Since $\mathrm{DM(RS)}$ is completely distributive, 
it forms a quasi-Nelson algebra.
Since it is also spatial, it is algebraic  as we already
noted. Hence, to prove (i) it is enough to show that the 
completely join-irreducible elements of $\mathrm{DM(RS)}$ satisfy 
the interpolation property. Now, assume (ii) and
suppose that $p,q\leq g(p),g(q)$ holds for some $p,q\in\mathcal{J}$. 
We will show that there exists an element $k\in\mathcal{J}$, such that
\[
p,q\leq k\leq g(p),g(q).
\]
Observe that we may exclude the cases $p=g(p)$ and $q=g(q)$.
For instance, if $p = g(p)$, then $k = p$ and the interpolation 
property holds.
 
Now $p,q<g(p),g(q)$ implies that $p,q\in\mathcal{J}^{-}$. 
By Proposition~\ref{prop:partition_of_J} this means that 
$p =(\emptyset,\{x\}^\UP)$ and $q=(\emptyset,\{y\}^\UP)$ for some 
$\{x\}^\UP$ and $\{y\}^\UP$ that are completely join-irreducible in
$\wp(U)^\UP$. 
Then $(\emptyset,\{x\}^\UP),(\emptyset,\{y\}^\UP)
\leq g(\emptyset,\{x\}^\UP),g(\emptyset,\{y\}^\UP)$ 
yields that that there exist $u\in U$ with $\{x\}^\UP,\{y\}^\UP 
\subseteq \{u\}^\UP$ by Lemma~\ref{lem:pre_interpolation}.
By (ii), there exist an element $z\in U$ such that 
$\{z\}^\UP$ is a completely join-irreducible in $\wp(U)^\UP$ and
$\{x\}^\UP,\{y\}^\UP \subseteq \{z\}^\UP$. 

If $z\notin \mathcal{S}$, then according to 
Proposition~\ref{prop:partition_of_J},
$k = (\emptyset,\{z\}^\UP) \in \mathcal{J}^{-}$. Thus, $p,q\leq k$.
This also implies $k \leq g(k) \leq g(p),g(q)$. 
If $z\in \mathcal{S}$, then $k=(\{z\},\{z\}^\UP)\in
\mathcal{J}^{\circ}$ and $p,q\leq k$. In addition,
$k = g(k) \leq g(p),g(q)$. We have now show that 
$\mathcal{J}$ satisfies the interpolation property.
\end{proof}

Next we provide a couple of examples
showing how Theorem~\ref{thm:characterization}
is used.

\begin{example} \label{Exa:NelsonEquivalence}
Let us continue considering Example~\ref{exa:JoinRS}.
In this example, we observed that $\mathrm{DM(RS)}$ forms a 
regular double Stone algebra. Consequently, it is also a Nelson
algebra, as can be demonstrated using our results.

In Example~\ref{exa:JoinRS}, we observed that $\wp(U)^\UP$ and 
$\wp(U)^\Up$ are finite Boolean lattices. Consequently, they are 
completely distributive and spatial. Thus, $\mathrm{DM(RS)}$ 
inherits these properties, allowing us to apply 
Theorem~\ref{thm:characterization}.

The elements $1$ and $3$ do not belong to $\mathcal{S}$. Both 
$\{1\}^\UP$ and $\{3\}^\UP$ are equal to $\{1,3\}$, and $\{1,3\}$ is 
completely join-irreducible in $\wp(U)^\UP$. Therefore, condition 
(ii) of Theorem~\ref{thm:characterization} holds, and 
$\mathrm{DM(RS)}$ forms a Nelson algebra.
\end{example}

\begin{example} \label{Exa:NelsonSimilar}
Let us recall from Example~\ref{exa:compl_distr} the relation
$R$ with the  $R$-neighbourhoods:
\begin{gather*}
R(1) = \{1,2\}, \ R(2) = \{1,2,3\}, \ R(3) = \{3\}, \ 
R(4) = \{1,3,4\}.
\end{gather*}
Clearly, the elements $1$, $2$, and
$4$ do not belong to $\mathcal{S}$. The sets $\{2\}^\UP$ and
$\{4\}^\UP$ are completely join-irreducible in $\wp(U)^\UP$.
They are both included in $\{1\}^\UP$, which is not completely
join-irreducible in $\wp(U)^\UP$. Since there is no element
$z$ such that  $\{z\}^\UP$ is completely join-irreducible in
$\wp(U)^\UP$ and both  $\{2\}^\UP$ and
$\{4\}^\UP$ are included in $\{z\}^\UP$, by 
Theorem~\ref{thm:characterization} we can state
that $\mathrm{DM(RS)}$ is not an Nelson algebra.

\begin{figure}
\includegraphics[width=50mm]{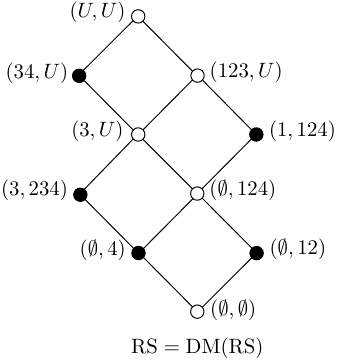}
\caption{The Hasse diagram of $\mathrm{RS} = \mathrm{DM(RS)}$.
Its completely join-irreducible elements are marked with
filled circles.}
\label{fig3:non_nelson}
\end{figure}

The Hasse diagram of $\mathrm{RS} = \mathrm{DM(RS)}$ is given in
Figure~\ref{fig3:non_nelson}. For simplicity, in the figure, 
we denote the subsets of $U$ that differ from  $\emptyset$ and 
$U$ by sequences of their elements. 
Let us consider in detail some set-theoretical calculations:
\begin{description}
\item[Union] The join of
$(\{4\}^\DOWN, \{4\}^\UP) = (\emptyset, \{4\})$ and
$(\{2\}^\DOWN, \{2\}^\UP) = (\emptyset, \{1,2\})$ is
\[ (\emptyset, \{4\}) \vee  (\emptyset, \{1,2\}) = (\emptyset, \{1,2,4\}),\]
which equals  $(\{2,4\}^\DOWN, \{2,4\}^\UP)$.

\item[Intersection] The meet of $(\{1,3,4\}^\DOWN, \{1,3,4\}^\UP) =
(\{3,4\},U)$ and $ (\{1,2,3\}^\DOWN, \{1,2,3\}^\UP) = (\{1,2,3\},U)$ is
\[ (\{3,4\},U) \wedge (\{1,2,3\},U) = (\{3\},U),  \]
which is the same as the rough set $(\{1,3\}^\DOWN, \{1,3\}^\UP)$.

\item[Complement] The De~Morgan complement of $(\{2,4\}^\DOWN, \{2,4\}^\UP) = (\emptyset, \{1,2,4\})$ is 
\[ {\sim} (\emptyset, \{1,2,4\}) =  (\{3\},U),\]
which is equal to $(\{1,3\}^\DOWN, \{1,3\}^\UP)$. 
\end{description}

\medskip\noindent%
Let us consider the map $g$. Now,
\begin{gather*}
g(\emptyset,\{4\}) = (\{3,4\},U), \quad
g(\emptyset,\{1,2\}) = (\{1\},\{124\}), \quad
g(\{3\},\{2,3,4\}) = (\{3\},\{2,3,4\}).
\end{gather*}
This means  that 
\begin{align*}
\mathcal{J}^- &= \{ (\emptyset,\{4\}), (\emptyset, \{1,2\}) \},\\
\mathcal{J}^\circ &= \{ (\{3\},\{2,3,4\}) \},\\
\mathcal{J}^+ &= \{ (\{3,4\}, U), (\{1\}, \{1,2\,4\}) \}.
\end{align*}
Because  $\mathrm{DM(RS)}$ does not form a Nelson algebra, the
interpolation property should not hold. This can seen by setting
$p = (\emptyset,\{4\})$ and $q = (\emptyset, \{1,2\})$.
Now $p,q < g(p),g(q)$, but there is no element $k \in \mathcal{J}$
such that $p,q \leq k \leq g(p),g(q)$.
\end{example}

The following proposition can be considered
as a corollary of Theorem~\ref{thm:characterization}.

\begin{proposition} \label{prop:impl_Nelson}
Let $R$ be a reflexive relation on $U$ such that
$\mathfrak{core}R(x)$ and $\mathfrak{core}\breve{R}(x)$ 
are nonempty for each $x \in U$, then 
$\mathrm{DM(RS)}$ forms a Nelson algebra.
\end{proposition}

\begin{proof} Let $R$ be a reflexive relation on $U$ such that
$\mathfrak{core}R(x)$ and $\mathfrak{core}\breve{R}(x)$ 
are nonempty for each $x \in U$.
Then each $R(x)$ is completely join-prime in $\wp(U)^\Up$ and
each $\breve{R}(x)$ is completely join-prime in $\wp(U)^\UP$
by the definition of the core. By
Corollary~\ref{cor:spatial_reflexive},
$\mathrm{DM(RS)}$ is completely distributive and spatial.

It is now clear that condition (ii) of Theorem~\ref{thm:characterization} is satisfied. Therefore,
$\mathrm{DM(RS)}$ forms a Nelson algebra.
\end{proof}

We end this section by presenting some remarks related to the 
previous studies \cite{ JR11, JarvRadeRivi24, JarRadVer09}.
Let $R$ be a quasiorder on $U$.
As we already mentioned in Remark~\ref{rem:alex} and 
Example~\ref{Exa:ThreeExamples}, 
$\wp(U)^\UP = \wp(U)^\Down$ and
$\wp(U)^\Up = \wp(U)^\DOWN$ form completely distributive spatial
lattices.
In addition, $\{ \{x\}^\UP \}_{x \in U}$ forms the 
set of completely join-irreducible elements of $\wp(U)^\UP$.
Because $\wp(U)^\UP$ is completely distributive, these elements are also completely join-prime.
The same observations hold for $\{\{x\}^\Up\}_{x\in U}$ and 
$\wp(U)^\Up$. 
Therefore, both $R(x)$ and $\breve{R}(x)$ have nonempty core
for each $x \in U$.
By Proposition~\ref{prop:impl_Nelson}, $\mathrm{DM(RS)}$ forms a 
Nelson algebra. It is proved already in \cite{JarRadVer09} 
that $\mathrm{RS}$ is itself a complete lattice. Thus,
$\mathrm{DM(RS)} = \mathrm{RS}$. Note also that we proved in 
\cite{JR11} that $\mathrm{RS}$ forms a 
Nelson algebra defined on an algebraic lattice.
It is also clear that if $R$ is a quasiorder, also $\breve{R}$ is
a quasiorder. This means that $\mathrm{RS}$ defined by $\breve{R}$
have the same lattice-structural properties as the rough set
lattice defined by $R$.

Moreover, for all $x \in U$,
\[ \mathfrak{core}R(x) = \{ w \mid x \, R \, w \text{ and }
w \, R \, x\}.\]
First note that $x \, R \, y$ is equivalent to $R(y) \subseteq R(x)$.
Let $w$ be such that $x \, R \, w$ and $w \, R \, x$. 
If $w \in R(y)$, then $y \,R \, w$. Now $w \, R \, x$ gives
$y \, R \, x$ by transitivity and hence $R(x) \subseteq R(y)$.
This means $w \in \mathfrak{core}R(x)$. 
On the other hand, $w \in \mathfrak{core}R(x)$ 
gives $x \in R(x) \subseteq R(w)$ by Lemma~\ref{lem:core}(v),
that is, $w \, R \, x$. By the definition of core,
$w \in R(x)$, that is, $x \, R \, w$.

\medskip%
However, in case of just reflexive relation, the situation
is different. It is clear that $R$ is reflexive if and only if 
$\breve{R}$ is reflexive. Let us denote by $\mathrm{DM(RS)}\breve{\;}$ 
the completion of rough sets defined by $\breve{R}$. We can now write
the following proposition.

\begin{proposition}
Let $R$ be a reflexive relation on $U$. The following observations hold.
\begin{enumerate}[label={\rm (\roman*)}]
\item  $\mathrm{DM(RS)}\breve{\;}$ is completely distributive 
if and only if $\mathrm{DM(RS)}$ is completely distributive. 
\item  $\mathrm{DM(RS)}\breve{\;}$ is spatial if and only if
 $\mathrm{DM(RS)}$ is spatial.
\item If $\mathfrak{core}R(x)$ and $\mathfrak{core}\breve{R}(x)$ 
are nonempty for each $x \in U$, then 
$\mathrm{DM(RS)}$ and $\mathrm{DM(RS)}\breve{\;}$
form Nelson algebras.
\end{enumerate}
\end{proposition}

\begin{proof}
(i) Let  $\mathrm{DM(RS)}$ be completely distributive. 
This is equivalent to the fact that any of $\wp(U)^\UP$, $\wp(U)^\DOWN$,  $\wp(U)^\Up$, $\wp(U)^\Down$ 
is completely distributive by Proposition~\ref{prop:completely_distributive}.
This is then equivalent to  $\mathrm{DM(RS)}\breve{\;}$
being completely distributive.

(ii) Suppose  $\mathrm{DM(RS)}$ is spatial. Then, by 
Proposition~\ref{prop:spatial}, this is equivalent to that
$\wp(U)^\UP$ and $\wp(U)^\Up$ are spatial. But this is again
equivalent to the situation that $\mathrm{DM(RS)}\breve{\;}$ is 
spatial.

Claim (iii) is clear by Proposition~\ref{prop:impl_Nelson}.
\end{proof}

\section{Some concluding remarks}

In this work, we have especially studied the rough set systems defined by reflexive binary relations. We have shown how reflexive relations can be interpreted as directional similarities. Even though $\mathrm{RS}$ defined by a reflexive relation does not form a lattice, the completion $\mathrm{DM(RS)}$ always is a lattice. Therefore, studying its lattice-theoretical properties opens a new view to rough set theory. In particular, we can perform computations on rough sets because their ``union'' ``intersection'', and ``completion'' always exist, as we have shown in Examples \ref{Exa:set_operations} and \ref{Exa:NelsonSimilar}.

In conclusion, we highlight a significant property of rough sets induced by equivalences or quasi-orders: they enable the definition of Nelson algebras. These algebras, which can be viewed as enriched Heyting algebras or as a subvariety of residuated lattices, allow the application of both intuitionistic logic-based and substructural logic-based deduction methods.

Until now, such rich algebraic structures were not known in the case of general binary relations. Our results show that, under certain conditions, even in the case of reflexive binary relations, the induced rough set systems form Nelson algebras. This makes it possible to extend known deduction methods to more general rough set systems. An example of such a rough set lattice can be seen in Example~\ref{Exa:NelsonEquivalence}.

\section*{Acknowledgements}
We extend our sincere gratitude to the two anonymous referees for their insightful comments and suggestions, which significantly enhanced the quality of this paper.
 
\bibliographystyle{abbrv}
\bibliography{bibliography}
\end{document}